\documentclass{article}
\usepackage{amssymb,amsmath,bm,paralist,graphics,epsfig,graphicx,epstopdf,amsthm,geometry,CJKutf8,color,yhmath}
\usepackage{cite}
\usepackage[colorlinks,urlcolor=blue]{hyperref}

\theoremstyle{plain}
\newtheorem{theorem}{Theorem}[section]

\newtheorem{notation}[theorem]{Notation}

\newcommand{\norm}[1]{\left\Vert#1\right\Vert}
\newcommand{\brac}[1]{\left(#1\right)}
\newcommand{\abs}[1]{\left\vert#1\right\vert}

\newcommand{\diag}{\mbox{diag}}
\newcommand{\diff}{\mbox{d}}

\usepackage{authblk} 
\usepackage{hyperref} 
\allowdisplaybreaks

\begin{document}

\baselineskip=1pc  

\begin{center}
{\bf \large A thermodynamics-based turbulence model for isothermal compressible flows}
\end{center}

\vspace{.2in}

\centerline{
Zhiting Ma \footnote{Beijing Institute of Mathematical Sciences and Applications, Beijing 101408, China. E-mail: mazt@bimsa.cn.}
\qquad
Wen-An Yong 
\footnotemark[1]\textsuperscript{,}
\footnote{Department of Mathematical Sciences, Tsinghua University, Beijing 100084, China. E-mail: wayong@tsinghua.edu.cn.}\textsuperscript{,}
{\renewcommand{\thefootnote}{\fnsymbol{footnote}}
\footnote{Corresponding author.}
}
\qquad
Yi Zhu 
\footnotemark[1]\textsuperscript{,}
\footnote[3]{Yau Mathematical Sciences Center, Tsinghua University, Beijing 100084, China.
E-mail: yizhu@tsinghua.edu.cn.}
}

\vspace{.4in}

\centerline{\bf Abstract}

\vspace{.2in}

This study presents a new turbulence model for isothermal compressible flows. The model is derived by combining the Favre averaging and the Conservation-dissipation formalism --- a newly developed thermodynamics theory. The latter provides a systematic methodology to construct closure relations that intrinsically satisfy the first and second laws of thermodynamics. The new model is a hyperbolic system of first-order partial differential equations. It has a number of numerical advantages, and addresses some drawbacks of classical turbulence models by resolving the non-physical infinite information propagation paradox of the parabolic-type models and accurately capturing the interaction between compressibility and turbulence dissipation. Furthermore, we show the compatibility of the proposed model with Prandtl's one-equation model for incompressible flows by deliberately rescaling the model and studying its low Mach number limit.

\vfill


\vspace{.2in}
%



\vspace{.2in}

\begin{CJK}{UTF8}{gbsn}


\section{Introduction}

Turbulent flows are ubiquitous in nature and almost all flows of practical engineering interest are turbulent. Typical examples include flows around moving objects such as airplanes, ships and automobiles. They are also important parts in a variety of transport phenomena occurring in the atmosphere and ocean. A correct understanding and accurate simulation of turbulence has important practical and scientific significance. Starting from the pioneering works of Reynolds, Prandtl, Kolmogorov, Boussinesq and von K\'arm\'an \cite{Reynolds1976, Prandtl1925, Kolmogorov1941, boussinesq1877, Karman1948}, many theoretical and computational models have been developed, such as Prandtl’s one-equation model \cite{Prandtl1945}, two-equations models \cite{hanjalic1972reynolds,launder1974numerical,saffman1970model,wilcox1988reassessment} and so on \cite{ahmadi1989two,huang1995compressible}. Most of the earlier works are for incompressible fluids \cite{Launder1972, Lumley1978, Rodi1982, Jones1972, Jones1973, Launder1975} and are based on the Reynolds averaging and the Boussinesq eddy-viscosity approximation. These models are quite successful in many applications \cite{launder1972mathematical, Bradshaw1981, deck2002development, menter2003ten}. However, they may give predictions significantly different from measurements \cite{bradshaw1973effects}.
 
The compressibility effects can substantially influence flow behaviors \cite{Jones1980, LibbyWilliams1980, sarkar1991analysis}. The extension of turbulence modeling to compressible flows introduces additional complexities arising from density variations, shock-turbulence interactions, and dilatational dissipation mechanisms \cite{sarkar1991analysis,durbi2018developments}. Notably, the classical Boussinesq hypothesis becomes increasingly problematic at higher Mach numbers due to the breakdown of equilibrium between the turbulent stress and mean strain rate \cite{ahmadi1989two,huang1995compressible}. Furthermore, the coupling between acoustic modes and turbulent kinetic energy transfer creates new closure challenges that are absent in incompressible formulations \cite{Blaisdell1993compressibility}. These inherent limitations motivate the need for a fundamentally new modeling paradigm grounded in thermodynamic consistency.

The goal of this work is to contribute a new perspective to this actively developing field. The core is to make full use of the first and second laws of thermodynamics and propose novel turbulence models for compressible flows. Towards this goal, we start with conservation laws for isothermal flows, take the Reynolds averaging or Favre averaging and obtain new conservation equations with a Reynolds-stress tensor. Then a Boussinesq-like approximation is adapted for the Reynolds stress. More importantly, the resultant conservation laws are closed by utilizing the Conservation-Dissipation Formalism (CDF) \cite{zhu2015conservation,yong2020cdf} -- a recently developed nonequilibrium thermodynamics theory, which provides a systematic methodology to construct closure relations that intrinsically satisfy the first and second laws of thermodynamics. The derived model is a set of first-order partial differential equations, which is hyperbolic and thermodynamically stable. In particular, it guarantees that the turbulent entropy production remains non-negative, thereby eliminating unphysical solutions.


Unlike the traditional turbulence models such as Prandtl’s one-equation model and the $k-\varepsilon$ model \cite{hanjalic1972reynolds,launder1974numerical}, our hyperbolic model does not contain second-order spatial derivatives in the governing equations. This fundamental architectural innovation possesses significant advantages. Physically, it inherently satisfies the causality principle in turbulent flow evolution and resolves the non-physical infinite information propagation paradox of the parabolic-type models.
Numerically, the hyperbolic nature of our first-order model enables efficient discretizations \cite{mcdonald2011extended, mcdonald2014application} via, for instance, the discontinuous Galerkin method \cite{cockburn2000development} and finite-volume methods \cite{leVeque2002Finite}. It also offers geometric flexibility through relaxed mesh quality requirements -- particularly advantageous for complex industrial configurations involving multi-element airfoils or labyrinthine cooling channels.

On the other hand, our new turbulence model is also different from the first-order hyperbolic models recently proposed in \cite{yan2020hyperbolic,yan2023hyperbolic}. Those models were obtained by mathematically relaxing the traditional second-order convection-diffusion turbulence models (including Prandtl’s one-equation model, the $k-\varepsilon$ and $k-\omega$ models), and by replacing the Navier-Stokes equations with the ten-moment equations \cite{levermore1998gaussian}. 
Although the models achieve hyperbolicity through mathematical relaxation techniques (e.g., introducing auxiliary variables for turbulent quantities), their approach lacks a clear thermodynamic basis. 
 In contrast, our CDF-based closure rigorously enforces the entropy inequality through its constitutive structure, ensuring thermodynamic consistency.

To show the reasonableness of our new turbulence model, we deliberately rescale the model to analyse its low Mach number limit. Particularly, it is justified that the new model is compatible with Prandtl's one-equation model for incompressible flows. The compatibility demonstrates that the new hyperbolic model successfully bridges the gap in the theoretical disparities between compressible and incompressible turbulent flow regimes, establishing a unifying paradigm that harmonizes traditionally distinct domains of fluid mechanics research.


The rest of the paper is organized as follows. Section \ref{sec2:rafe} presents the Favre-averaged fluid equations and the Boussinesq-like approximation for the corresponding Reynolds stress. In Section \ref{sec3:CDF} we introduce the Conservation-dissipation formalism and construct a toy model to show how CDF works.  Our new hyperbolic turbulence model is constructed in Section \ref{sec4:cdf-model}. In Section \ref{sec:5-compatibility} we rescale the new turbulence model. The low Mach number limit is justified in Section \ref{sec:6-justification}.

\section{Favre-averaged Fluid Equations}\label{sec2:rafe}

Consider an isothermal compressible fluid. Without external forces, the conservation laws of mass and momentum for such electrically neutral fluids read as
\begin{equation}\label{equ:fluids-equations}
	\begin{aligned}
		&\partial_t \rho  + \nabla\cdot (\rho \bm{u} ) = 0,\\
		&\partial_t (\rho \bm{u}) + \nabla\cdot (\rho \bm{u} \otimes \bm{u}  +  \bm{P})  = 0.
	\end{aligned}
\end{equation}
Here $\rho$ is the fluid density, $\bm{u} \in \mathbb{R}^3$ is the velocity, $\bm{P}$ is the stress tensor which is symmertric, $\nabla$ is the gradient operator with respect to the space variable $\bm{x}=(x_1, x_2, x_3)$, $\otimes$ denotes the tensorial product. $(x, t) \in \Omega \times (0, +\infty) $ with $\Omega \subset \mathbb{R}^3$.
This system governs the evolution of variables $(\rho, \rho \bm{u})$, while the stress tensor $\bm{P}$ has to be modeled for the above equations to be closed.

In turbulence modeling for incompressible flows, it is a common practice to decompose the fluid variables into mean and fluctuating parts: 
\begin{equation}\label{equ:variable-decompose1}
    \begin{aligned}
        \rho = \bar{\rho} + \rho^\prime, \quad \bm{P} = \bar{\bm{P}} + \bm{P}^\prime.
    \end{aligned}
\end{equation}
Here, $\bar{\rho}$ and $\rho^\prime$ denote the Reynolds average (mean) and fluctuation of $\rho$, respectively; and $\bar{\bm{P}}$ and $\bm{P}^\prime$ have similar meanings.
By the nature of the Reynolds average, we have the following fundamental properties \cite{tennekes1972first}:
\begin{equation}\label{equ:average-operator}
    \begin{aligned}
       & \overline{\phi^\prime} =0,\quad \overline{\bar{\phi} \phi^\prime} = \overline{\bar{\phi} \psi^\prime} = 0, \quad  \overline{\phi \psi} = \bar{\phi} \bar{\psi} + \overline{\phi^\prime \psi^\prime}, \quad \overline{\phi^2} = \bar{\phi}^2 + \overline{\phi'^2},\\[3mm]
       & \overline{\frac{\partial \phi}{\partial t}} = \frac{\partial \overline{\phi}}{\partial t}, \quad \overline{\frac{\partial \phi}{\partial x_i}} = \frac{\partial \overline{\phi}}{\partial x_i},\quad (i=1, 2, 3),\\[3mm]
       & \overline{\phi \psi \xi}=\bar{\phi} \bar{\psi} \bar{\xi}+\overline{\phi^{\prime} \psi^{\prime}} \bar{\xi}+\overline{\phi^{\prime} \xi^{\prime}} \bar{\psi}+\overline{\psi^{\prime} \xi^{\prime}} \bar{\phi}+\overline{\phi^{\prime} \psi^{\prime} \xi^{\prime}}.
    \end{aligned}
\end{equation}
Here $\phi$ and $\psi$ are two generic fluid variables.
 
For compressible flows, the Reynolds average is often enhanced by the Favre average (or mass-weighted average) \cite{favre1969statistical}, defined as
\begin{equation}\nonumber
	\tilde{\phi} = \overline{\rho \phi}/\bar{\rho}
\end{equation}
for $\bar{\rho}>0$.
By this definition, the fluid variable is decomposed into
\begin{equation}
	\begin{aligned}\label{equ:variable-decompose2}
		\phi = \tilde{\phi} + \phi^{\prime\prime}
	\end{aligned}
\end{equation}
and the following properties hold:
\begin{equation}\nonumber
	\overline{\rho \phi} = \bar{\rho} \tilde{\phi},\quad \overline{\rho \phi^{\prime\prime}} = 0, 
	\quad \overline{\rho \phi \psi} = \bar{\rho} \tilde{\phi}\tilde{\psi} + \overline{\rho \phi^{\prime\prime} \psi^{\prime\prime}}, 
	\quad \overline{\phi \tilde{\psi}}=\bar{\phi}\tilde{\psi}. 
\end{equation}
It is important to notice that $\overline{\phi^{\prime\prime}}$ may not vanish while $\overline{\phi^{\prime}}=0$.
Note that the Favre average $\tilde{\bm{u}}$ of velocity $\bm{u}$ is identical to its Reynolds average $\bar{\bm{u}}$ for incompressible flows. 

By taking the Reynolds average on the equations \eqref{equ:fluids-equations} and using the Favre average of the fluid velocity, we obtain the following conservation laws 
\begin{equation}\label{equ:rans-cdf}
\begin{aligned}
&\partial_t \bar{\rho}  + \nabla\cdot (\bar{\rho} \bm{\tilde{u}} ) = 0,\\
&\partial_t (\bar{\rho} \bm{\tilde{u}}) + \nabla\cdot (\bar{\rho} \bm{\tilde{u}} \otimes \bm{\tilde{u}} + \bar{\rho}\widetilde{ \bm{u}^{\prime\prime}\otimes \bm{u}^{\prime\prime} }+ \bm{\bar{P}})  = 0.
\end{aligned}
\end{equation}
This system is not closed since it contains the mean stress tensor $\bm{\bar{P}}$ and additional unknown tensor
\begin{equation}\nonumber
	\bm{\tau} \equiv - \bar{\rho}\widetilde{ \bm{u}^{\prime\prime}\otimes \bm{u}^{\prime\prime} }.
\end{equation}
The symmetric semi-negative definite tensor $\bm{\tau}$ is the Reynolds stress that incorporates the effects of turbulent motions on the mean stresses.

The closure problem is to express the mean stress tensor $\bm{\bar{P}}$ and the Reynolds stress tensor $\bm{\tau}$ as functions of the mean-field and other variables.
To do this, we recall the traditional way to close the system \eqref{equ:fluids-equations} for laminar flows, where the stress tensor $\bm{P}$ is modelled by using Newton’s law of viscosity.
This empirical law, together with the conservation laws \eqref{equ:fluids-equations}, forms the classical Navier-Stokes equations.
However, these equations are insufficient in some situations such as fluids with memory or high-frequency heat transfer, where flux or stress relaxation times are significant, as in polymer solutions and suspensions \cite{jou1996extended}. To address such non-classical scenarios, the Conservation-dissipation formalism (CDF) was developed  \cite{zhu2015conservation}. 

For turbulent flows, the Reynolds-stress tensor $\bm{\tau}$ can be viewed as a second-order moment of velocity fluctuations. 
Thus we can use the conservation laws \eqref{equ:fluids-equations} to derive its governing equation, which depends on third-order moments of velocity fluctuations, due to the nonlinearity in equations \eqref{equ:fluids-equations}.
Moreover, the governed equations of third-order moments can be derived and depend on fourth-order moments, and so on. 
In this way, we derive a hierarchy of infinitely many moment equations. 
On physical grounds, this is not surprising, since the operations are strictly mathematical in nature and introduce no additional physical principles. 
 In essence, the Reynolds averaging is a brutal simplification that causes the loss of much of the information contained in the conservation laws. 
The aim of turbulence modelling is to express high-order moments in terms of lower-order moments \cite{wilcox1998turbulence}. 

Another approach is to approximate the Reynolds-stress tensor $\bm{\tau}$ with the so-called Boussinesq approximation \cite{boussinesq1877}
\begin{equation}\label{appro:Boussinesq}
\begin{aligned}
\bm{\tau} = \nu_T \brac{\nabla \bm{\tilde{u}} + (\nabla \bm{\tilde{u}})^T} - \frac{2}{3} k \bm{I}_3.
\end{aligned}
\end{equation}
Here $\nu_T$ is the turbulence eddy viscosity, $k=\frac{1}{2}\bar{\rho} \widetilde{\bm{u}^{\prime\prime}\cdot\bm{u}^{\prime\prime}}$ is the turbulent kinetic energy related to the trace of $\bm{\tau}$, $\bm{I}_3$ is the second-order unit tensor in $\mathbb{R}^3$ and the superscript $T$ denotes the transpose.
This approximation has been used widely in the simulation of practical turbulent flows, although it is not unsatisfactory for compressible flows.

Based on the Boussinesq approximation, several eddy-viscosity models have been developed, and typical examples are one-equation models and two-equation models \cite{wilcox1998turbulence}. 
These models all contain the following turbulent kinetic energy equation:
\begin{equation}\label{equ:sec2-k-equ}
\partial_t  k + \bm{\tilde{u}}\cdot \nabla k = \bm{\tau}:\nabla \bm{\tilde{u}} - \epsilon + \nabla \cdot \brac{ \left(\nu+\nu_T\right) \nabla k },
\end{equation}
where $\epsilon$ is the dissipation rate and $\nu$ is the shear viscosity. Here the colon stands for the double contraction of two second-order tensors: $\bm{A}: \bm{B}=\sum_{i, j} A_{i j} B_{j i}$.
The last equation was derived by taking the trace of the traditional Reynolds-stress equation \cite{wilcox1998turbulence}. It is a second-order partial differential equation, while the original conservation laws \eqref{equ:fluids-equations} contain only first-order derivatives.
Other deficiencies of the Boussinesq approximation can be found in \cite{alfonsi2009reynolds}.

In an attempt to overcome the deficiencies, we write the mean stress tensor as
\begin{equation}\nonumber
    \bar{\bm{P}} = \bar{\bm{\sigma}} + \bar{p} \bm{I}_3
\end{equation}
with the thermodynamic pressure $\bar{p}$. Note that $\bar{\bm{\sigma}}$ is the averaged deviatoric stress tensor which is symmetric.
With $\bar{\bm{\sigma}}$ thus defined, 
we modify the Boussinesq approximation \eqref{appro:Boussinesq} as
\begin{equation}\label{defi:CDF-boussinesq}
\bm{\tau}= -\mu_T \bar{\bm{\sigma}} - \frac{2}{3} k \bm{I}_3
\end{equation}
with $\mu_T$ for the turbulence viscosity.
With this modification, the Reynolds-averaged equations \eqref{equ:rans-cdf} can be rewritten as 
\begin{equation}\label{equ:key-equ}
\begin{aligned}
&\partial_t \bar{\rho}  + \nabla\cdot (\bar{\rho} \bm{\tilde{u}} ) = 0,\\
&\partial_t (\bar{\rho} \bm{\tilde{u}}) + \nabla\cdot \brac{\bar{\rho} \bm{\tilde{u}} \otimes \bm{\tilde{u}} + (\mu_T+1)\bar{\bm{\sigma}} } + \nabla \brac{\bar{p}+\frac{2}{3}k } = 0.
\end{aligned}
\end{equation}
Here the unclosed terms are $\bar{\bm{\sigma}}$, $k$ and $\mu_T$. 
In this work, we simply take $\mu_T = \nu_T / \nu$ and turbulence eddy viscosity $\nu_T=k^{1/2} l$ with $l$ the length scale of the turbulence as in many one-equation turbulence models \cite{wilcox1998turbulence}.

\section{Conservation-dissipation formalism}\label{sec3:CDF}

Here we review the Conservation-dissipation formalism (CDF) developed in \cite{zhu2015conservation} to close the Reynolds-averaged conservation equations \eqref{equ:key-equ}. 
It is a quite new theory of nonequilibrium thermodynamics and is based on the first and second laws thereof. 
The theory is concerned with an irreversible process which obeys some conservation laws like \eqref{equ:key-equ}:
\begin{equation}\label{conservation-laws}
	\partial_t \bm{u} + \sum_{j=1}^3 \partial_{x_j}\bm{f}_j = 0.
\end{equation}	
Here $\bm{u} = \bm{u}(\bm{x}, t) \in \mathbb{R}^n$ represents conserved variables and $\bm{f}_j$ is the corresponding flux along the $x_j$-direction. 
If each $\bm{f}_j$ is given in terms of the conserved variables, the system \eqref{conservation-laws} becomes closed. In this case, the system is considered to be in local equilibrium, and $\bm{u}$ is also referred to as equilibrium variables. 
However, very often $\bm{f}_j$ depends on some extra variables in addition to the conserved ones.
The extra variables characterize non-equilibrium features of the system under consideration, called non-equilibrium or dissipative variables, and their choice is not unique which depends on the physical and mathematical structures of the resultant systems.

Inspired by the Extended Thermodynamics \cite{muller1998rational,jou1996extended}, CDF chooses a set of dissipative variable $\bm{v} \in \mathbb{R}^r$, so the flux $\bm{f}_j$ in \eqref{conservation-laws} can be expressed as $\bm{f}_j=\bm{f}_j(\bm{u}, \bm{v})$ and seeks evolution equations of the form
\begin{equation}\nonumber
	\partial_t \bm{v}+\sum_{j=1}^3 \partial_{x_j} \bm{g}_j(\bm{u}, \bm{v})=\bm{q}(\bm{u}, \bm{v}).
\end{equation}
Here $\bm{g}_j(\bm{u}, \bm{v})$ is the corresponding flux and $\bm{q}=\bm{q}(\bm{u}, \bm{v})$ is the nonzero source, vanishing at equilibrium. 
Together with the conservation laws \eqref{conservation-laws}, the evolution of a non-equilibrium state is then governed by the following system of first-order PDEs
\begin{equation}\label{CDF-PDE}
	\partial_t \bm{U}+\sum_{j=1}^3 \partial_{x_j} \bm{F}_j(\bm{U})=\bm{Q}(\bm{U}),
\end{equation}
where
\begin{equation}\nonumber
	\begin{aligned}
		\bm{U}=\begin{pmatrix}
			\bm{u} \\ \bm{v}
		\end{pmatrix}, \quad \bm{F}_j(\bm{U})=\begin{pmatrix}
			\bm{f}_j(\bm{U}) \\ \bm{g}_j(\bm{U})
		\end{pmatrix}, \quad \bm{Q}(\bm{U})=\begin{pmatrix}
			\bm{0} \\ \bm{q}(\bm{U})
		\end{pmatrix}.
	\end{aligned}
\end{equation}

Furthermore, CDF requires the balance laws \eqref{CDF-PDE} to respect the {\it{Conservation-dissipation Principle}} proposed in \cite{yong2008interesting}:\begin{itemize}
	\item[(i)] There is a strictly concave smooth function $\eta=\eta(\bm{U})$, called entropy, such that the matrix product $\eta_{\bm{U} \bm{U}} \bm{F}_{j \bm{U}}(\bm{U})$ is symmetric for each $j$ and for all $\bm{U}$ under consideration.
	\item[(ii)] There is a positive definite matrix $\bm{M}=\bm{M}(\bm{U})$, called dissipation matrix, such that the non-zero source can be written as $\bm{q}(\bm{U})=\bm{M}(\bm{U}) \eta_{\bm{v}}(\bm{U})$.
\end{itemize}
Here the subscript stands for the corresponding partial derivative, for instance, $\eta_{\bm{v}}=\frac{\partial \eta}{\partial \bm{v}}$ and $\eta_{\bm{U} \bm{U}}=\frac{\partial^2 \eta}{\partial \bm{U}^2}$, and $\eta_v(\bm{U})$ should be understood as a column vector. Note that the dissipation matrix is not assumed to be symmetric and its positive definiteness means that of the symmetric part $\frac{\bm{M}+\bm{M}^T}{2}$. This principle was shown in \cite{yong2008interesting} to be a strengthened version of the structural stability conditions proposed in \cite{yong1999singular,yong2004entropy} for hyperbolic systems of PDEs with relaxation.

Balance laws \eqref{CDF-PDE} together with the conservation-dissipation principle will be referred to as conservation-dissipation formalism (CDF). This formalism has two freedoms: the entropy function $\eta=\eta(\bm{U})$ and the dissipation matrix $\bm{M}=\bm{M}(\bm{U})$. They are both functions of the state variable $\bm{U}$. The former is strictly concave and the latter is positive definite. Except these, no further restriction is imposed on $\eta=\eta(\bm{U})$ or $\bm{M}=\bm{M}(\bm{U})$. Specific expressions of $\eta(\bm{U})$ and $\bm{M}(\bm{U})$ should be problem-dependent.

Here are some simple comments on the conservation-dissipation principle. Condition (i) is the well-known entropy condition for hyperbolic conservation laws \cite{godunov1961interesting,friedrichs1971systems}. It corresponds to the classical thermodynamics stability. This condition ensures that the first-order system \eqref{CDF-PDE} is globally symmetrizable hyperbolic and thereby well-posed \cite{majda1984compressible}. It implies that there is a function $\bm{J}_j=\bm{J}_j(\bm{U})$ such that
\begin{equation}\label{cdf-sys}
	\eta_{\bm{U}} \cdot \bm{F}_{j \bm{U}}=\bm{J}_{j \bm{U}} .
\end{equation}
This imposes a restriction on the flux $\bm{g}_j(\bm{u}, \bm{v})$. Moreover, we can use equations \eqref{CDF-PDE} and \eqref{cdf-sys} to compute the rate of change of entropy:
\begin{equation}\nonumber
\begin{aligned}
\partial_t \eta & =-\sum_{j=1}^3 \eta_{\bm{U}} \cdot \partial_{x_j} \bm{F}_j+\eta_{\bm{U}} \cdot \bm{q} \\
& =- \nabla \cdot \bm{J} + \bm{H}
\end{aligned}
\end{equation}
with the entropy production $\bm{H}=\eta_{\bm{U}} \cdot \bm{M}(\bm{U}) \eta_{\bm{U}} \geq 0$. Here Condition (ii) has been used. Thus, the second law of thermodynamics and thereby the time-irreversibility are respected automatically by the CDF-based system \eqref{CDF-PDE}.
In addition, Condition (ii) can be viewed as a nonlinearization of the celebrated Onsager reciprocal relation for scalar processes \cite{yong2020cdf,de1984non}. 

We conclude this section by seeking closure equations for the deviatoric stress tensor $\bm{\bar{\sigma}}$ and turbulent kinetic energy $k$ in the conservation laws \eqref{equ:key-equ} to illustrate how CDF works. 
To do this, we follow \cite{zhu2015conservation}  and introduce two non-equilibrium state variables $\bm{C}$ and $w$ which have the respective sizes of the tensor $\bm{\bar{\sigma}}$ and $k$. 
Then we specify an entropy function of the form
\begin{equation}\nonumber
	\eta = \eta(\bar{\rho}, \bar{\rho} \bm{\tilde{u}}, \bar{\rho} \bm{C}, \bar{\rho} w) = \bar{\rho} s(1/\bar{\rho},  \bm{\tilde{u}}, \bm{C}, w).
\end{equation}
It was shown in \cite{yang2018generalized} that $\eta= \eta(\bar{\rho}, \bar{\rho} \bm{\tilde{u}}, \bar{\rho} \bm{C}, \bar{\rho} w)$ is strictly concave if and only if so is $s=s(v,  \bm{\tilde{u}}, \bm{C}, w)$ with $v=1/\bar{\rho}$.

In what follows, we choose $s=s(v,  \bm{\tilde{u}}, \bm{C}, w)$ (like that in \eqref{entropy-form} below) to be a strictly concave function and
\begin{equation}\nonumber
	s_{\tilde{\bm{u}}}=-\tilde{\bm{u}},
\end{equation}
which is consistent with the classical Gibbs relation for isothermal fluids.
Define
\begin{equation}\nonumber
	\bar{p} = \frac{\partial s}{\partial v}.
\end{equation}
It is not difficult to verify that 
\begin{equation}\label{equ:gibbs-rela}
\eta-\bar{p} = \bar{\rho} \eta_{\bar{\rho}} + \bar{\rho}\bm{\tilde{u}} \cdot \eta_{\bar{\rho}\bm{\tilde{u}}} + \bar{\rho} \bm{C} : \eta_{\bar{\rho} \bm{C}}^T + \bar{\rho}w \eta_{\bar{\rho}w}, \qquad \eta_{\bar{\rho}\bm{\tilde{u}}} = s_{\tilde{\bm{u}}}= -\bm{\tilde{u}}.
\end{equation}

With last relations, we use the equations \eqref{equ:key-equ} to compute the change rate of entropy:
\begin{equation}\nonumber
\begin{aligned}
\partial_t \eta ={}& \eta_{\bar{\rho}}\partial_t \bar{\rho} + \eta_{\bar{\rho}\bm{\tilde{u}}}\cdot \partial_t(\bar{\rho}\bm{\tilde{u}}) + \eta_{\bar{\rho} \bm{C}}^T:\partial_t(\bar{\rho} \bm{C}) + \eta_{\bar{\rho}w} \partial_t(\bar{\rho}w) \\[2mm]
={}& -\eta_{\bar{\rho}} \nabla\cdot (\bar{\rho}\bm{\tilde{u}}) -  \nabla \cdot \brac{\bar{\rho}\bm{\tilde{u}}\otimes\bm{\tilde{u}} + \bar{p} \bm{I} + \brac{\mu_T+1}\bar{\bm{\sigma}} +\frac{2}{3}k \bm{I} } \cdot \eta_{\bar{\rho}\bm{\tilde{u}}} + \eta_{\bar{\rho} \bm{C}}^T:\partial_t(\bar{\rho} \bm{C}) + \eta_{\bar{\rho}w} \partial_t(\bar{\rho}w)\\[2mm]
={}& -\bigg( \bar{\rho} \eta_{\bar{\rho}} + \bar{\rho}\bm{\tilde{u}} \cdot \eta_{\bar{\rho}\bm{\tilde{u}}} + \bar{\rho} \bm{C} : \eta_{\bar{\rho} \bm{C}}^T + \bar{\rho}w\eta_{\bar{\rho}w} \bigg) \nabla \cdot \bm{\tilde{u}} 
- \bm{\tilde{u}} \cdot \bigg( \eta_{\bar{\rho}} \nabla \bar{\rho} + \eta_{\bar{\rho}\bm{\tilde{u}}} \nabla(\bar{\rho}\bm{\tilde{u}}) + \eta_{\bar{\rho} \bm{C}}^T: \nabla(\bar{\rho} \bm{C})  + \eta_{\bar{\rho}w} \nabla(\bar{\rho}w) \bigg)  \\[2mm]
{}& - \nabla \cdot \brac{\bar{p} \bm{I} + \brac{\mu_T+1}\bar{\bm{\sigma}} +\frac{2}{3}k \bm{I}} \cdot \eta_{\bar{\rho}\bm{\tilde{u}}}
+ \eta_{\bar{\rho} \bm{C}}^T:\bigg(\partial_t(\bar{\rho} \bm{C}) + \nabla\cdot (\bar{\rho} \bm{C}\bm{\tilde{u}}) \bigg)  
+ \eta_{\bar{\rho}w} \bigg(\partial_t(\bar{\rho}w)+\nabla\cdot (\bar{\rho}w \bm{\tilde{u}}) \bigg)  \\[2mm]
={}& -(\eta - \bar{p})\nabla\cdot\bm{\tilde{u}} - \bm{\tilde{u}}\cdot \nabla \eta 
+ \nabla \cdot \brac{\bar{p} \bm{I} + \brac{\mu_T+1}\bar{\bm{\sigma}}+\frac{2}{3} k \bm{I}}\cdot \bm{\tilde{u}} \\[2mm]
{}&+ \eta_{\bar{\rho} \bm{C}}^T:\bigg(\partial_t(\bar{\rho} \bm{C}) + \nabla\cdot (\bar{\rho} \bm{C}\bm{\tilde{u}}) \bigg)  + \eta_{\bar{\rho}w} \bigg(\partial_t(\bar{\rho}w)+\nabla\cdot (\bar{\rho}w \bm{\tilde{u}}) \bigg)  \\[2mm]
={}& -\nabla\cdot\bigg(\bm{\tilde{u}} (\eta - \bar{p})\bigg)  
+ \nabla\cdot\bigg( \big(\brac{\mu_T+1}\bar{\bm{\sigma}} +\frac{2}{3} k \bm{I}\big)\cdot \bm{\tilde{u}} \bigg) 
- \brac{\brac{\mu_T+1}\bar{\bm{\sigma}}^T +\frac{2}{3} k \bm{I}}:\nabla\bm{\tilde{u}}\\[2mm]
{}&+ \eta_{\bar{\rho} \bm{C}}^T:\bigg(\partial_t(\bar{\rho} \bm{C}) + \nabla\cdot (\bar{\rho} \bm{C}\bm{\tilde{u}}) \bigg)  + \eta_{\bar{\rho}w} \bigg(\partial_t(\bar{\rho}w)+\nabla\cdot (\bar{\rho}w \bm{\tilde{u}}) \bigg)  \\[2mm]
={}& -\nabla \cdot \bigg((\eta - \bar{p} - \frac{2}{3} k)\bm{\tilde{u}}  - \brac{\mu_T+1}\bm{\bar{\sigma}}\cdot\bm{\tilde{u}} \bigg)  \\[2mm]
{}&+ \eta_{\bar{\rho} \bm{C}}^T:\bigg(\partial_t(\bar{\rho} \bm{C}) + \nabla \cdot (\bm{\tilde{u}} \bar{\rho} \bm{C}) \bigg) 
- \brac{\mu_T+1}\bm{\bar{\sigma}}^T:\nabla \bm{\tilde{u}} 
+ \eta_{\bar{\rho}w} \bigg(\partial_t(\bar{\rho}w)+\nabla \cdot  (\bar{\rho}w\bm{\tilde{u}}) \bigg)
- \frac{2}{3} k \nabla\cdot \bm{\tilde{u}}  \\[2mm]
={}& -\nabla \cdot \bm{J} + \bm{H}.
\end{aligned}
\end{equation}
Here $\bm{J} = (\eta - \bar{p} - \frac{2}{3}k)\bm{\tilde{u}}  - \brac{\mu_T+1}\bm{\bar{\sigma}}\cdot\bm{\tilde{u}} $ is the entropy flux and 
\begin{equation}\nonumber
    \begin{aligned}
        \bm{H} =\eta_{\bar{\rho} \bm{C}}^T:\bigg(\partial_t(\bar{\rho} \bm{C}) + \nabla \cdot (\bm{\tilde{u}} \bar{\rho} \bm{C}) \bigg) 
        - \brac{\mu_T+1}\bm{\bar{\sigma}}^T:\nabla \bm{\tilde{u}} 
        + \eta_{\bar{\rho}w} \bigg(\partial_t(\bar{\rho}w)+\nabla \cdot  (\bar{\rho}w\bm{\tilde{u}}) \bigg)
        - \frac{2}{3}k \nabla\cdot \bm{\tilde{u}}        
    \end{aligned}
\end{equation}
is the entropy production. 
Motivated by this expression of the entropy production, we refer to \cite{zhu2015conservation} and take 
\begin{equation}\nonumber
    (\mu_T+1) \bm{\bar{\sigma}} = \eta_{\bar{\rho}\bm{C}}, \quad  k = \eta_{\bar{\rho}w},
\end{equation}
meaning that the tensor $\bm{C}$ has been chosen to be symmetric due to the symmetry of $\bar{\bm{\sigma}}$.
Then, for the entropy production to be positive, we choose the evolution equations for the nonequilibrium variables $\bm{C}$ and $w$:
\begin{equation}\label{equ:cdf-evo-equ-0}
\begin{pmatrix}
      \partial_t (\bar{\rho}\bm{C}) + \nabla \cdot (\bar{\rho}\bm{\tilde{u}}\otimes \bm{C}) - \frac{1}{2}\brac{\nabla\bm{\tilde{u}} + (\nabla\bm{\tilde{u}})^T}\\[2mm]
      \partial_t (\bar{\rho}w) + \nabla \cdot (\bar{\rho}w\bm{\tilde{u}}) - \frac{2}{3} \nabla\cdot \bm{\tilde{u}} \\[2mm]
\end{pmatrix}
    = \bm{M} \begin{pmatrix}
         (\mu_T+1) \bm{\bar{\sigma}}\\[2mm]
         k\\[2mm]
    \end{pmatrix} 
\end{equation}
with $\bm{M}=\bm{M}(\bar{\rho},\bar{\rho} \bm{\tilde{u}},  \bar{\rho}\bm{C}, \bar{\rho}w) $ a positive definite matrix. Such a matrix can be chosen similarly as that in \eqref{choice-M-2} below. 
The last equations and the conservation laws \eqref{equ:key-equ} together form a closed system of $11$ equations as our toy model.

This toy model is globally symmetrizable hyperbolic and thereby well-posed \cite{majda1984compressible}.
But it does not appear to be compatible with the turbulent kinetic energy equation \eqref{equ:sec2-k-equ} and thereby not compatible with existing well-validated theories. This indicates that the toy model does not properly characterize the dissipation and production mechanisms for the turbulent kinetic energy.

\section{Hyperbolic turbulence models}\label{sec4:cdf-model}

In this section, we address the production and dissipation mechanisms of turbulent kinetic energy and propose a new hyperbolic turbulence model by closing the conservation laws \eqref{equ:key-equ} with the CDF.

In order to obtain a model compatible with Prandtl’s one-equation model, we introduce a new non-equilibrium variable $\bm{y}\in \mathbb{R}^3$ for the diffusion term in equation \eqref{equ:sec2-k-equ}. 
As in the previous section, we choose a strictly concave entropy function $s=s(1/\bar{\rho},  \bm{\tilde{u}}, \bm{C}, w, \bm{y})$ with $s_{\tilde{\bm{u}}}=-\tilde{\bm{u}}$ and take the entropy function
\begin{equation}\nonumber
\eta=\eta(\bar{\rho}, \bar{\rho} \bm{\tilde{u}}, \bar{\rho} \bm{C}, \bar{\rho}w,\bar{\rho} \bm{y})= \bar{\rho} s(1/\bar{\rho},  \bm{\tilde{u}}, \bm{C}, w, \bm{y}).
\end{equation}
As before, it is easy to show the following relations
\begin{equation}\nonumber
\eta-\bar{p} = \bar{\rho} \eta_{\bar{\rho}} + \bar{\rho}\bm{\tilde{u}} \cdot \eta_{\bar{\rho}\bm{\tilde{u}}} + \bar{\rho} \bm{C} : \eta_{\bar{\rho} \bm{C}}^T + \bar{\rho}w\eta_{\bar{\rho}w}+\bar{\rho}\bm{y}\cdot \eta_{\bar{\rho}\bm{y}}, \qquad \eta_{\bar{\rho}\bm{\tilde{u}}} = -\bm{\tilde{u}}.
\end{equation}

With these relations, we use the conservation laws  \eqref{equ:key-equ} and compute the change rate of entropy:
\begin{equation}\nonumber
\begin{aligned}
\partial_t \eta ={}& \eta_{\bar{\rho}}\partial_t \bar{\rho} + \eta_{\bar{\rho}\bm{\tilde{u}}}\cdot \partial_t(\bar{\rho}\bm{\tilde{u}}) + \eta_{\bar{\rho} \bm{C}}^T:\partial_t(\bar{\rho} \bm{C}) + \eta_{\bar{\rho}w} \partial_t(\bar{\rho}w)+\eta_{\bar{\rho}\bm{y}}\cdot \partial_t (\bar{\rho}\bm{y}) \\[2mm]
={}& -\eta_{\bar{\rho}} \nabla\cdot (\bar{\rho}\bm{\tilde{u}}) -  \nabla \cdot \brac{\bar{\rho}\bm{\tilde{u}}\otimes\bm{\tilde{u}} + \bar{p} \bm{I} + \brac{\mu_T+1}\bar{\bm{\sigma}} +\frac{2}{3}k \bm{I} } \cdot \eta_{\bar{\rho}\bm{\tilde{u}}} \\[2mm]
{}&+ \eta_{\bar{\rho} \bm{C}}^T:\partial_t(\bar{\rho} \bm{C}) + \eta_{\bar{\rho}w} \partial_t(\bar{\rho}w)+\eta_{\bar{\rho}\bm{y}}\cdot \partial_t (\bar{\rho}\bm{y})\\[2mm]
={}& -\bigg( \bar{\rho} \eta_{\bar{\rho}} + \bar{\rho}\bm{\tilde{u}} \cdot \eta_{\bar{\rho}\bm{\tilde{u}}} + \bar{\rho} \bm{C} : \eta_{\bar{\rho} \bm{C}}^T + \bar{\rho}w\eta_{\bar{\rho}w} + \bar{\rho}\bm{y}\cdot \eta_{\bar{\rho}\bm{y}} \bigg) \nabla \cdot \bm{\tilde{u}} \\[2mm]
{}&- \bigg( \eta_{\bar{\rho}} \nabla \bar{\rho} + \eta_{\bar{\rho}\bm{\tilde{u}}} \nabla (\bar{\rho}\bm{\tilde{u}}) + \eta_{\bar{\rho} \bm{C}}^T: \nabla (\bar{\rho} \bm{C})  + \eta_{\bar{\rho}w}\nabla (\bar{\rho}w)+\eta_{\bar{\rho}\bm{y}}\cdot (\nabla (\bar{\rho}\bm{y}))^T \bigg) \cdot\bm{\tilde{u}}  \\[2mm]
{}& - \nabla \cdot  \brac{\bar{p} \bm{I} + \brac{\mu_T+1}\bm{\bar{\sigma}} +\frac{2}{3}k \bm{I}} \cdot \eta_{\bar{\rho}\bm{\tilde{u}}}
+ \eta_{\bar{\rho} \bm{C}}^T:\bigg(\partial_t(\bar{\rho} \bm{C}) + \nabla\cdot (\bar{\rho}\bm{\tilde{u}} \bm{C}) \bigg)  
+ \eta_{\bar{\rho}w} \bigg(\partial_t(\bar{\rho}w)+\nabla\cdot (\bar{\rho}\bm{\tilde{u}}w) \bigg)  \\[2mm]
{}& + \bigg( \partial_t(\bar{\rho}\bm{y})+\nabla\cdot(\bar{\rho}\bm{\tilde{u}}\otimes \bm{y}) \bigg)\cdot  \eta_{\bar{\rho}\bm{y}}\\[2mm]
={}& -(\eta - \bar{p})\nabla\cdot\bm{\tilde{u}} - \bm{\tilde{u}}\cdot\nabla \eta 
+ \nabla\cdot \brac{\bar{p} \bm{I} + \brac{\mu_T+1}\bm{\bar{\sigma}} +\frac{2}{3}k \bm{I}}\cdot \bm{\tilde{u}} \\[2mm]
{}&+ \eta_{\bar{\rho} \bm{C}}^T:\bigg(\partial_t(\bar{\rho} \bm{C}) + \nabla\cdot (\bar{\rho}\bm{\tilde{u}} \bm{C}) \bigg)  + \eta_{\bar{\rho}w} \bigg(\partial_t(\bar{\rho}w)+\nabla\cdot(\bar{\rho}\bm{\tilde{u}}w) \bigg) + \bigg( \partial_t(\bar{\rho}\bm{y})+\nabla\cdot(\bar{\rho}\bm{\tilde{u}}\otimes\bm{y}) \bigg)\cdot \eta_{\bar{\rho}\bm{y}} \\[2mm]
={}& -\nabla \cdot \bigg((\eta-\bar{p}-\frac{2}{3}k)\bm{\tilde{u}}  - \brac{\mu_T+1}\bm{\bar{\sigma}}\cdot\bm{\tilde{u}} +  \xi k \eta_{\bar{\rho}\bm{y}} \bigg) + \eta_{\bar{\rho} \bm{C}}^T:\bigg(\partial_t(\bar{\rho} \bm{C}) + \nabla \cdot ( \bar{\rho}\bm{\tilde{u}} \bm{C}) \bigg) 
- \brac{\mu_T+1}\bm{\bar{\sigma}}^T:\nabla \bm{\tilde{u}} \\[2mm]
{}&+ \eta_{\bar{\rho}w} \bigg(\partial_t(\bar{\rho}w)+\nabla \cdot  (\bar{\rho}\bm{\tilde{u}}w) \bigg)-\frac{2}{3}k \nabla\cdot \bm{\tilde{u}} +\xi k\nabla\cdot \eta_{\bar{\rho}\bm{y}}   +  \bigg( \partial_t(\bar{\rho}\bm{y})+\nabla\cdot(\bar{\rho}\bm{\tilde{u}}\otimes\bm{y}) \bigg)\cdot \eta_{\bar{\rho}\bm{y}} + \xi \eta_{\bar{\rho}\bm{y}}\cdot \nabla k \\[2mm]
={}& -\nabla \cdot \bm{J} + \bm{H}.
\end{aligned}
\end{equation}
Here $\bm{J} = (\eta - \bar{p} - \frac{2}{3} k)\bm{\tilde{u}}  - \brac{\mu_T+1}\bm{\bar{\sigma}}^T\cdot\bm{\tilde{u}} +\xi  k \eta_{\bar{\rho}\bm{y}} $ is the entropy flux and 
\begin{equation}\label{relation:entropy production}
\begin{aligned}
    \bm{H} ={}&\eta_{\bar{\rho} \bm{C}}^T:\bigg(\partial_t(\bar{\rho} \bm{C}) + \nabla \cdot (\bar{\rho}\bm{\tilde{u}}  \bm{C}) \bigg) - \brac{\mu_T+1}\bm{\bar{\sigma}}^T:\nabla \bm{\tilde{u}}\\[2mm] 
    {}&+ \eta_{\bar{\rho}w} \bigg(\partial_t(\bar{\rho}w)+\nabla \cdot  (\bar{\rho}\bm{\tilde{u}}w) \bigg)-\frac{2}{3} k \nabla\cdot \bm{\tilde{u}} + \xi k\nabla\cdot \eta_{\bar{\rho}\bm{y}}\\[2mm]
    {}&+ \bigg( \partial_t(\bar{\rho}\bm{y})+\nabla\cdot(\bar{\rho}\bm{\tilde{u}}\otimes\bm{y}) \bigg)\cdot \eta_{\bar{\rho}\bm{y}} + \xi \eta_{\bar{\rho}\bm{y}}\cdot \nabla k
\end{aligned}
\end{equation}
is the entropy production with $\xi$ a positive constant. 

Motivated by the last expression of the entropy production, we refer to Section \ref{sec3:CDF} and take 
\begin{equation}\nonumber
    (\mu_T+1) \bm{\bar{\sigma}} = \eta_{\bar{\rho}\bm{C}}, \quad  k = \eta_{\bar{\rho}w}.
\end{equation}
Then, for the entropy production to be positive, we choose evolution equations for the nonequilibrium variables $\bm{C}$, $w$ and $\bm{y}$ as
\begin{equation}\label{equ:cdf-evo-equ-1}
\begin{pmatrix}
      \partial_t (\bar{\rho}\bm{C}) + \nabla \cdot (\bar{\rho}\bm{\tilde{u}}\bm{C}) - \frac{1}{2}\brac{\nabla\bm{\tilde{u}} + (\nabla\bm{\tilde{u}})^T}\\[2mm]
      \partial_t (\bar{\rho}w) + \nabla \cdot (\bar{\rho}\bm{\tilde{u}}w) - \frac{2}{3} \nabla\cdot \bm{\tilde{u}} + \xi \nabla \cdot \eta_{\bar{\rho}\bm{y}} \\[2mm]
       \partial_t (\bar{\rho} \bm{y}) + \nabla \cdot (\bar{\rho}\bm{\tilde{u}}\otimes \bm{y}) + \xi \nabla k \\[2mm]
\end{pmatrix}
    = \bm{M}\begin{pmatrix}
         (\mu_T+1) \bm{\bar{\sigma}}\\[2mm]
         k\\[2mm]
         \eta_{\bar{\rho}\bm{y}}\\[2mm]
    \end{pmatrix} 
\end{equation}
with $\bm{M}=\bm{M}(\bar{\rho},\bar{\rho} \bm{\tilde{u}},  \bar{\rho}\bm{C}, \bar{\rho}w,  \bar{\rho}\bm{y})$ positive definite. 
The last equations and the conservation laws \eqref{equ:key-equ} together form a closed system of $14$ equations as our hyperbolic turbulence model.

Next, we present some specific choices of the entropy function and dissipation matrix. 
Given a strictly concave function $s^{eq}=s^{eq}(v)$ as a specific entropy for an isothermal flow in equilibrium,
we take $\eta = \bar{\rho} s$ and 
\begin{equation}\label{entropy-form}
s = s^{eq}(1/\bar{\rho})  - \frac{|\bm{\tilde{u}}|^2}{2}  - \frac{|\bm{C} |^2}{2\alpha_1}- \frac{w^2}{2\alpha_2}- \frac{|\bm{y}|^2}{2\alpha_3}
\end{equation}
with $\alpha_1, \alpha_2$ and $\alpha_3$ three positive parameters. This function is obviously strictly concave and satisfies $s_{\tilde{\bm{u}}} = -\tilde{\bm{u}}$.
With this choice, we have 
\begin{equation}\nonumber
(\mu_T+1)\bar{\bm{\sigma}}=\eta_{\bar{\rho}\bm{C}} =-\frac{\bm{C}} {\alpha_1}, \quad k = \eta_{\bar{\rho}w}=-\frac{ w}{\alpha_2}, \quad \eta_{\bar{\rho}\bm{y}} =-\frac{\bm{y}}{\alpha_3}.
\end{equation}
Hence the equation \eqref{equ:cdf-evo-equ-1} become
\begin{equation}\label{equ:cdf-evo-equ}
\begin{pmatrix}
   \alpha_1\left[\partial_t\left((\mu_T+1)\bar{\rho} \bar{\bm{\sigma}}\right)+\nabla \cdot\left( (\mu_T+1)\bar{\rho} \tilde{\bm{u}} \bar{\bm{\sigma}}\right)\right]+ \frac{1}{2}\brac{\nabla\bm{\tilde{u}} + (\nabla\bm{\tilde{u}})^T}\\[2mm]
    \alpha_2\left[ \partial_t (\bar{\rho} k) + \nabla \cdot (\bar{\rho} k\tilde{\bm{u}}) \right] + \frac{2}{3} \nabla\cdot \bm{\tilde{u}} + \frac{1}{\alpha_3} \nabla \cdot \bm{y} \\[2mm]
     -\partial_t (\bar{\rho} \bm{y}) - \nabla \cdot (\bar{\rho}\bm{\tilde{u}}\otimes \bm{y}) - \nabla k \\[2mm]
\end{pmatrix}
    = -\bm{M}\begin{pmatrix}
         (\mu_T+1) \bm{\bar{\sigma}}\\[2mm]
         k\\[2mm]
         \eta_{\bar{\rho}\bm{y}}\\[2mm]
    \end{pmatrix}. 
\end{equation}

In order to choose the dissipation matrix $\bm{M}$, we review Prandtl’s one-equation model \cite{Prandtl1945}, which is 
just the turbulent kinetic energy equation \eqref{equ:sec2-k-equ} with the dissipation rate $\epsilon = C_D \frac{k^{3 / 2}}{l}$ and $C_D$ being a closure coefficient:
\begin{equation}\label{equ:k-equ-prandtl}
	\partial_t  k + \bm{\tilde{u}}\cdot \nabla k = \bm{\tau}:\nabla \bm{\tilde{u}} - C_D \frac{k^{3/2}}{l} + \nabla \cdot \brac{ \left(\nu+\nu_T\right) \nabla k }
\end{equation}
for an incompressible flow.
Denote by $\bm{S}$ the mean strain-rate tensor:
\begin{equation}\nonumber
	\bm{S}=\frac{1}{2}\brac{\nabla \bm{\tilde{u}} + (\nabla \bm{\tilde{u}})^T} .
\end{equation}
With the Boussinesq approximation \eqref{appro:Boussinesq} for the Reynolds stress tensor $\bm{\tau}$, the production term in \eqref{equ:k-equ-prandtl} can be reformulated as
\begin{equation}\nonumber
\begin{aligned}
	\bm{\tau}:\nabla \bm{\tilde{u}}\equiv {}\brac{2\nu_T\bm{S} - \frac{2}{3}k\bm{I}_3}:\nabla \bm{\tilde{u}}={}2\nu_T\bm{S}:\nabla \bm{\tilde{u}}-\frac{2}{3} k \underbrace{\nabla \cdot \bm{\tilde{u}}}_{=0}={}2\nu_T\bm{S}:\frac{1}{2}\brac{\nabla \bm{\tilde{u}} + (\nabla \bm{\tilde{u}})^T}={}2\nu_T\bm{S}:\bm{S}
\end{aligned}
\end{equation}
for incompressible flows where $\nabla \cdot \bm{\tilde{u}}=0$.
Thus, the one-equation \eqref{equ:k-equ-prandtl} can be rewritten as 
\begin{equation}\label{equ:prandtl-one-equ}
	\partial_t  k + \bm{\tilde{u}}\cdot \nabla k = 2\nu_T\bm{S}:\bm{S} - C_D \frac{k^{3/2}}{l} + \nabla \cdot \brac{ \left(\nu+\nu_T\right) \nabla k }.
\end{equation}

For the new constitutive equations \eqref{equ:cdf-evo-equ} to be compatible with \eqref{equ:prandtl-one-equ}, we choose the dissipation matrix $\bm{M}$ such that
\begin{equation}\label{choice-M}
\bm{M} \begin{pmatrix}
         (\mu_T+1) \bar{\bm{\sigma}}\\[2mm]
         k\\[2mm]
        \eta_{\bar{\rho}\bm{y}}
	\end{pmatrix} = \begin{pmatrix}
    	\frac{\bar{\rho} \bar{\bm{\sigma}}}{\nu}\\[2mm]
    	-\frac{2 \beta \bar{\rho} \nu_T}{\nu^2} \bar{\bm{\sigma}} : \bar{\bm{\sigma}} + \frac{\beta C_D\bar{\rho}}{l}k^{3/2} \\[2mm]
    	\frac{\bar{\rho} \eta_{\bar{\rho}\bm{y}}}{\nu + \nu_T}
	\end{pmatrix} 
\end{equation}
with $\beta$ a positive constant.
Obviously, the symmetric matrix  \begin{equation}\label{choice-M-2}
	\begin{aligned}
		\bm{M} = \begin{pmatrix}
			 \frac{\bar{\rho}}{\nu (\mu_T +1)}\brac{1 + \frac{2\beta \nu_T k}{\nu}}\bm{I}_6 & - \frac{2 \beta\bar{\rho} \nu_T}{\nu^2(\mu_T+1)}\bar{\bm{\sigma}} & 0\\[3mm]
			-\frac{2\beta \bar{\rho} \nu_T}{\nu^2(\mu_T+1)}\bar{\bm{\sigma}}^T &C_D\beta\bar{\rho} \frac{\nu_T}{l^2} & 0\\[3mm]
			0 & 0 & \frac{\bar{\rho} }{\nu + \nu_T}\bm{I}_3
		\end{pmatrix},
	\end{aligned}
\end{equation}
ensures \eqref{choice-M} and is positive definite under the condition
\begin{equation}\label{constraint-k}
	 \frac{4 \beta\nu_T l^2}{\nu^2(\nu_T+\nu)}|\bar{\bm{\sigma}}|^2 < C_D\brac{1 + \frac{2\beta\nu_T}{\nu}k}.
\end{equation}
Consequently, the evolution equations \eqref{equ:cdf-evo-equ} become
\begin{equation}\label{equ:sht-model}
\begin{aligned}
& \alpha_1\left[\partial_t\left((\mu_T+1)\bar{\rho} \bar{\bm{\sigma}}\right)+\nabla \cdot\left( (\mu_T+1)\bar{\rho} \tilde{\bm{u}} \bar{\bm{\sigma}}\right)\right]+ \frac{1}{2}\brac{\nabla\bm{\tilde{u}} + (\nabla\bm{\tilde{u}})^T}=-\frac{\bar{\rho}\bar{\bm{\sigma}}}{\nu}, \\[2mm]
&\alpha_2\left[ \partial_t (\bar{\rho} k) + \nabla \cdot (\bar{\rho} \tilde{\bm{u}}k) \right] + \frac{2}{3} \nabla\cdot \bm{\tilde{u}} + \frac{\xi}{\alpha_3} \nabla \cdot \bm{y} = \frac{2\beta\bar{\rho}\nu_T}{\nu^2} \bar{\bm{\sigma}} :\bar{\bm{\sigma}} - \beta C_D \frac{\bar{\rho} k^{3/2}}{l}, \\[2mm]
&\partial_t (\bar{\rho} \bm{y}) + \nabla \cdot (\bar{\rho}\tilde{\bm{u}}\otimes \bm{y}) + \xi\nabla k 
 = -\frac{\bar{\rho}\bm{y}}{\alpha_3(\nu + \nu_T)}.
\end{aligned}
\end{equation}

The compatibility of the last equations with \eqref{equ:prandtl-one-equ} can be easily seen as follows.  
When the relaxation parameters $\alpha_1$ and $\alpha_3$ are small, we apply the Maxwell iteration \cite{muller1998rational} to \eqref{equ:sht-model} and obtain 
\begin{equation}\nonumber
\begin{aligned}
	\bar{\rho}\bar{\bm{\sigma}} \approx - \frac{\nu}{2}\brac{\nabla\bm{\tilde{u}} + (\nabla\bm{\tilde{u}})^T}=-\nu \bm{S}, \quad  \bar{\rho}\bm{y} \approx - \xi \alpha_3(\nu + \nu_T) \nabla k.	
\end{aligned}
\end{equation}
Substituting these approximations of $\bar{\rho}\bar{\bm{\sigma}}$ and $\bar{\rho}\bm{y}$, 
together with incompressibility $\bar{\rho}=1$ and $\nabla \cdot \bm{\tilde{u}}=0$,  into the second equation in \eqref{equ:sht-model} with $\alpha_2=\xi= \beta =1$ simply gives Prandtl's one-equation \eqref{equ:prandtl-one-equ}.

\section{Scaling Analysis}\label{sec:5-compatibility}

In this section, we rescale the constructed model for low Mach number flows.
First of all, we write the conservation laws \eqref{equ:key-equ} and the constitutive  relations \eqref{equ:sht-model} together as
\begin{equation}\label{equ:full-equ}
\begin{aligned}
& \partial_t \rho +\nabla \cdot (\rho \bm{u} )=0, \\[2mm]
& \partial_t \brac{ \rho\bm{u} } + \nabla \cdot \left(\rho\bm{u} \otimes \bm{u} + \brac{\mu_T+1} \bm{\sigma} \right) + \nabla \brac{p+ \frac{2}{3}k} =0, \\[2mm]
& \alpha_1 \big[ \partial_t\left( (\mu_T+1)\rho \bm{\sigma}\right)+\nabla \cdot\left( (\mu_T+1) \rho \bm{u}\bm{\sigma}\right)\big] +\frac{1}{2}\brac{\nabla\bm{u} + (\nabla\bm{u})^T}=- \frac{\rho\bm{\sigma}}{\nu}, \\[2mm]
&\alpha_2 \big[ \partial_t (\rho k) + \nabla \cdot (\rho k\bm{u}) \big] + \frac{2}{3} \nabla\cdot \bm{u} + \frac{ \xi}{\alpha_3} \nabla \cdot \bm{y} = \frac{2 \beta\rho\nu_T}{ \nu^2} \bm{\sigma} : \bm{\sigma} - \beta C_D \frac{\rho k^{3/2}}{l}, \\[2mm]
&\partial_t (\rho \bm{y}) + \nabla \cdot (\rho\bm{u}\otimes \bm{y}) + \xi \nabla k = - \frac{\rho\bm{y}}{\alpha_3(\nu + \nu_T)}.
\end{aligned}
\end{equation}
Here the bars and tildes in \eqref{equ:key-equ} and \eqref{equ:sht-model} have been omitted.

As in \cite{feireisl2007,alazard2006low}, we let $\varepsilon$ be a small positive parameter characterising the Mach number and define
\begin{equation}\nonumber
	\begin{aligned}
		&\rho^\varepsilon(x, \tau )=\rho(x, \frac{\tau}{\varepsilon}),\quad \bm{u}^\varepsilon(x, \tau) =\frac{1}{\varepsilon}\bm{u}(x, \frac{\tau}{\varepsilon}), \\[2mm]
		&k^\varepsilon(x, \tau) =\frac{1}{\varepsilon^2} k(x, \frac{\tau}{\varepsilon}), \quad \bm{y}^\varepsilon(x, \tau) = \frac{1}{\sqrt{\varepsilon}}\bm{y}(x, \frac{\tau}{\varepsilon}),\quad \tilde{\bm{\sigma}}^\varepsilon(x, \tau) =\frac{1}{\varepsilon^{3/2}}\bm{\sigma}(x, \frac{\tau}{\varepsilon}).
	\end{aligned}
\end{equation}
Correspondingly, we rescale the parameters in \eqref{equ:full-equ} as
\begin{equation}\nonumber
	\begin{aligned}
	&\nu_T^\varepsilon \equiv l\sqrt{k^\varepsilon} = \frac{l\sqrt{k}}{\varepsilon}\equiv \frac{\nu_T}{\varepsilon},\quad \bar{\nu}=\frac{\nu}{\varepsilon}, \quad \mu_T^\varepsilon \equiv \frac{\nu_T^\varepsilon}{\bar{\nu}} = \mu_T, \\[2mm]
	&\bar{\alpha}_1 = \varepsilon\alpha_1,\quad \bar{\alpha}_2 = \varepsilon^2\alpha_2,\quad \bar{\alpha}_3 = \varepsilon\alpha_3,\\[2mm]
	&\bar{\xi} = \varepsilon \xi, \quad \bar{\beta}=\varepsilon^2 \beta.
	\end{aligned}
\end{equation}
With such a scaling, the model can be rewritten as 
\begin{equation}\nonumber
\begin{aligned}
& \partial_\tau \rho^\varepsilon +\nabla \cdot (\rho^\varepsilon \bm{u}^\varepsilon )=0, \\[2mm]
& \partial_\tau \brac{\rho^\varepsilon\bm{u}^\varepsilon} + \nabla \cdot \brac{\rho^\varepsilon\bm{u}^\varepsilon \otimes \bm{u}^\varepsilon}+ \frac{1}{\sqrt{\varepsilon}} \nabla \cdot \brac{ \brac{\mu_T^\varepsilon +1} \bm{\tilde{\sigma}}^\varepsilon } + \frac{1}{\varepsilon^2} \nabla p^\varepsilon + \frac{2}{3} \nabla k^\varepsilon =0, \\[2mm]
& \bar{\alpha}_1 \big[\partial_\tau\left( (\mu_T^\varepsilon+1) \rho^\varepsilon\bm{\tilde{\sigma}}^\varepsilon\right)+\nabla \cdot\left( (\mu_T^\varepsilon+1) \rho^\varepsilon\bm{u}^\varepsilon  \bm{\tilde{\sigma}}^\varepsilon\right) \big] + \frac{1}{2\sqrt{\varepsilon}}\brac{\nabla\bm{u}^\varepsilon + (\nabla\bm{u}^\varepsilon)^T} = -\frac{\rho^\varepsilon\bm{\tilde{\sigma}}^\varepsilon}{\varepsilon\bar{\nu}}, \\[2mm]
&\bar{\alpha}_2\big[\partial_\tau (\rho^\varepsilon k^\varepsilon) + \nabla \cdot (\rho^\varepsilon \bm{u}^\varepsilon k^\varepsilon) \big]+ \frac{2}{3}\nabla\cdot \bm{u}^\varepsilon + \frac{\bar{\xi}}{\sqrt{\varepsilon}\bar{\alpha}_3} \nabla \cdot \bm{y}^\varepsilon =  \frac{2\bar{\beta}\rho^\varepsilon \nu_T^\varepsilon}{\varepsilon\bar{\nu}^2} \bm{\tilde{\sigma}}^\varepsilon : \bm{\tilde{\sigma}}^\varepsilon -\bar{\beta} C_D \frac{\rho^\varepsilon(k^\varepsilon)^{3/2}}{l}, \\[2mm]
&\partial_\tau (\rho^\varepsilon \bm{y}^\varepsilon) + \nabla \cdot (\rho^\varepsilon\bm{u}^\varepsilon\otimes \bm{y}^\varepsilon) +\frac{\bar{\xi}}{\sqrt{\varepsilon}}  \nabla k^\varepsilon = -\frac{\rho^\varepsilon\bm{y}^\varepsilon}{\varepsilon\bar{\alpha}_3(\nu_T^\varepsilon + \bar{\nu})}. 
\end{aligned}
\end{equation}

Following \cite{majda1984compressible,yong2005note}, we introduce
\begin{equation}\nonumber
	\phi^\varepsilon=\left(p^\varepsilon-p_0 \right) / \varepsilon
\end{equation}
with $p_0=p(\rho_0)>0$. Since $p = p(\rho)$ is strictly increasing, it has an inverse $\rho=\rho(p)$. Set $q(p)=\left[\rho(p) p^{\prime}(\rho(p))\right]^{-1}$.
Moreover, we set $\bm{\sigma}^\varepsilon =(\mu_T+1) \bm{\tilde{\sigma}}^\varepsilon$ and omit bars in the last system.
Thus, the last system can be rewritten as
\begin{equation}\label{equ:4-4-fullequ}
\begin{aligned}
& \partial_\tau \phi^\varepsilon + \bm{u}^\varepsilon \cdot\nabla \phi^\varepsilon  +  \frac{1}{\varepsilon q^\varepsilon} \nabla \cdot \bm{u}^\varepsilon =0, \\[2mm]
&\partial_\tau \bm{u}^\varepsilon + \bm{u}^\varepsilon \cdot \nabla\bm{u}^\varepsilon + \frac{1}{\sqrt{\varepsilon} \rho^\varepsilon }\nabla\cdot \bm{\sigma}^\varepsilon + \frac{1}{\varepsilon \rho^\varepsilon} \nabla \phi^\varepsilon  + \frac{2}{3 \rho^\varepsilon} \nabla k^\varepsilon =0, \\[2mm]
&\partial_\tau \bm{\sigma}^\varepsilon +\bm{u}^\varepsilon \cdot \nabla \bm{\sigma}^\varepsilon + \frac{1}{\sqrt{\varepsilon} \alpha_1\rho^\varepsilon} \frac{1}{2}\big(\nabla\bm{u}^\varepsilon + (\nabla\bm{u}^\varepsilon)^T\big) = -\frac{1}{\varepsilon  \alpha_1}\frac{\bm{\sigma}^\varepsilon}{\nu_T^\varepsilon + \nu}, \\[2mm]
&\partial_\tau k^\varepsilon + \bm{u}^\varepsilon \cdot \nabla k^\varepsilon + \frac{2}{3\alpha_2\rho^\varepsilon}\nabla\cdot \bm{u}^\varepsilon + \frac{\xi}{\sqrt{\varepsilon}\alpha_2\alpha_3\rho^\varepsilon} \nabla \cdot \bm{y}^\varepsilon = \frac{2 \beta}{\varepsilon\alpha_2}\frac{\nu_T^\varepsilon}{(\nu_T^\varepsilon + \nu)^2} \bm{\sigma}^\varepsilon : \bm{\sigma}^\varepsilon - \frac{\beta C_D}{\alpha_2 l}(k^\varepsilon)^{3/2}, \\[2mm]
&\partial_\tau \bm{y}^\varepsilon + \bm{u}^\varepsilon \cdot \nabla \bm{y}^\varepsilon + \frac{\xi}{\sqrt{\varepsilon} \rho^\varepsilon}\nabla k^\varepsilon = -\frac{1}{\varepsilon } \frac{\bm{y}^\varepsilon}{\alpha_3(\nu_T^\varepsilon + \nu)}.
\end{aligned}
\end{equation}
Here $q^\varepsilon = q(p_0 + \varepsilon \phi^\varepsilon)$ and $\rho^\varepsilon = \rho(p_0 + \varepsilon \phi^\varepsilon)$.

Without loss of generality, we write
\begin{equation}\nonumber
	\bm{\sigma}^\varepsilon=\begin{pmatrix}
		\sigma_{11} & \sigma_{12} & \sigma_{13}\\
		\sigma_{12} & \sigma_{22} & \sigma_{23}\\
		\sigma_{13} & \sigma_{23} & \sigma_{33}\\
	\end{pmatrix}
\end{equation} 
and set $\bm{U}^\varepsilon = (\phi^\varepsilon, ~ (\bm{u}^\varepsilon)^T, ~\sigma_{11},~\sigma_{12},~\sigma_{13},~\sigma_{22},~\sigma_{23},~\sigma_{33}, ~k^\varepsilon, ~(\bm{y}^\varepsilon)^T)^T$. 
Thus, we can rewrite \eqref{equ:4-4-fullequ} in the vector form
\begin{equation}\label{formal-full-equ-1}
	\bm{U}^\varepsilon_{\tau} + \sum_{j=1}^D \bm{A}_j(\bm{U}^\varepsilon)\bm{U}^\varepsilon_{x_j} = \bm{Q}(\bm{U}^\varepsilon).
\end{equation}
Here
\begin{equation}\label{equ:Aj-define}
\begin{aligned}
	\sum_{j=1}^3\bm{A}_j(\bm{U}^\varepsilon) \zeta_j &= (\bm{u}^\varepsilon \cdot \bm{\zeta}) \bm{I}_{14} + \begin{pmatrix}
		0 & \frac{1}{\varepsilon q^\varepsilon}\bm{\zeta} & \bm{0}_{1\times 6} & 0 & \bm{0}_{1\times 3}\\[2mm]
		\frac{1}{\varepsilon \rho^\varepsilon}\bm{\zeta}^T & \bm{0}_{3\times 3} & \frac{1}{\sqrt{\varepsilon}\rho^\varepsilon}\bm{C}_{3\times 6} & \frac{2}{3\rho^\varepsilon}\bm{\zeta}^T & \bm{0}_{3\times 3}\\[2mm]
		\bm{0}_{6\times 1} & \frac{1}{\sqrt{\varepsilon}\alpha_1\rho^\varepsilon}\bm{D}_{6\times 3} & \bm{0}_{6\times 6}  & \bm{0}_{6\times 1} & \bm{0}_{6\times 3}\\[2mm]
		0 & \frac{2}{3\alpha_2\rho^\varepsilon} \bm{\zeta} & \bm{0}_{1\times 6} & 0 & \frac{\xi}{\sqrt{\varepsilon}\alpha_2\alpha_3\rho^\varepsilon} \bm{\zeta}^T\\[2mm]
		\bm{0}_{3\times 1} & \bm{0}_{3\times 3} & \bm{0}_{3\times 6} & \frac{\xi}{\sqrt{\varepsilon}\rho^\varepsilon} \bm{\zeta}^T & \bm{0}_{3\times 3}
	\end{pmatrix}\\[4mm]
	&\equiv (\bm{u}^\varepsilon \cdot \bm{\zeta}) \bm{I}_{14} + \sum_{j=1}^3\bm{A}_j^{(1)} (\varepsilon \phi^\varepsilon) \zeta_j + \frac{1}{\sqrt{\varepsilon}}\sum_{j=1}^3\bm{A}_j^{(2)}(\varepsilon \phi^\varepsilon)\zeta_j + \frac{1}{\varepsilon}\sum_{j=1}^3\bm{A}_j^{(3)}(\varepsilon \phi^\varepsilon)\zeta_j 
\end{aligned}
\end{equation}
with $\bm{\zeta} = (\zeta_1, ~\zeta_2, ~\zeta_3)\in \mathcal{R}^3$, 
\begin{equation}\nonumber
\begin{aligned}
	\bm{C}_{3\times 6} = \begin{pmatrix}
		\zeta_1 & \zeta_2 & \zeta_3 & 0 & 0 & 0 \\
		0 & \zeta_1 & 0 & \zeta_2 & \zeta_3 & 0 \\
		0 & 0 & \zeta_1 & 0 & \zeta_2 & \zeta_3 \\
	\end{pmatrix}, 
	\quad \bm{D}_{6\times 3} = \begin{pmatrix}
		\zeta_1 & 0 & 0 \\[2mm]
		\frac{1}{2}\zeta_2 & \frac{1}{2}\zeta_1 & 0 \\[2mm]
		\frac{1}{2}\zeta_3 & 0 & \frac{1}{2}\zeta_1 \\[2mm]
		0 & \zeta_2 & 0 \\[2mm]
		0 & \frac{1}{2}\zeta_3 & \frac{1}{2}\zeta_2 \\[2mm]
		0 & 0 & \zeta_3 \\[2mm]
	\end{pmatrix},	
\end{aligned}
\end{equation}
 and
\begin{equation}\nonumber
\begin{aligned}
	\bm{Q}(\bm{U}^\varepsilon) = \begin{pmatrix}
		0 \\ \bm{0}_{3\times 1} \\[3mm] -\frac{1}{\varepsilon  \alpha_1}\frac{\bm{\sigma}^\varepsilon}{\nu_T^\varepsilon +\nu}\\[3mm] 
		\frac{2 \beta}{\varepsilon\alpha_2} \frac{\nu_T}{(\nu_T^\varepsilon+\nu)^2} \bm{\sigma}^\varepsilon : \bm{\sigma}^\varepsilon - \frac{\beta C_D}{\alpha_2 l} (k^\varepsilon)^{3/2}\\[3mm]
		-\frac{1}{\varepsilon \alpha_3} \frac{\bm{y}^\varepsilon}{\nu_T^\varepsilon+\nu}
	\end{pmatrix}
\end{aligned}
\end{equation}
with $\bm{\sigma}^\varepsilon$ in the third row of $\bm{Q}(\bm{U}^\varepsilon)$ being $(\sigma_{11},~\sigma_{12},~\sigma_{13},~\sigma_{22},~\sigma_{23},~\sigma_{33})^T$.

Suppose that, as $\varepsilon$ goes to zero, the limits
\begin{equation}\nonumber
	\phi^\varepsilon /\varepsilon \rightarrow \pi, \quad\bm{u}^\varepsilon\rightarrow \bm{u}, \quad\bm{\sigma}^\varepsilon/\sqrt{\varepsilon} \rightarrow \bm{\sigma} , \quad k^\varepsilon\rightarrow k,\quad \bm{y}^\varepsilon/\sqrt{\varepsilon} \rightarrow \bm{y}
\end{equation}
exist. We can easily see the following equations from the rescaled hyperbolic turbulence model \eqref{equ:4-4-fullequ} 
\begin{equation}\label{equ:rans-prandtl-one}
\begin{aligned}
& \nabla \cdot \bm{u} =0, \\
& \rho_0(\partial_\tau \bm{u} + \bm{u}\cdot \nabla \bm{u}) + \nabla \cdot \bm{\sigma} + \nabla \pi +  \frac{2}{3}\nabla k =0, \\[2mm]
&\rho_0 \bm{\sigma} = -\frac{1}{2}(\nu_T+\nu)\big(\nabla\bm{u} + (\nabla\bm{u})^T\big),\\
&\partial_\tau  k + \bm{u}\cdot \nabla k  + \frac{\xi}{\alpha_2\alpha_3\rho_0} \nabla \cdot \bm{y} =\frac{2 \beta}{\alpha_2} \frac{\nu_T}{(\nu_T+\nu)^2} \bm{\sigma} : \bm{\sigma} - \frac{\beta C_D}{\alpha_2 l}  (k)^{3/2}, \\[2mm]
&\rho_0\bm{y} = -\xi\alpha_3(\nu_T+\nu) \nabla k.
\end{aligned}
\end{equation}
With $\rho_0=1$, $\alpha_2=\beta =\xi^2$, and $\alpha_3$ being any positive constant, the last system of equations is just the incompressible Reynolds-averaged Navier–Stokes equations coupled with Prandtl's one equation \eqref{equ:k-equ-prandtl}.

\section{Rigorous justification}\label{sec:6-justification}
In this section, we justify the formal derivation of the classical one-equation turbulence model \eqref{equ:rans-prandtl-one} from the rescaled hyperbolic turbulence model \eqref{formal-full-equ-1}.
Notice that the rescaled model is a symmetrizable hyperbolic system with symmetrizer
\begin{equation}\label{defi:ao-symmetrizer}
	\bm{A}_0(\bm{U}^\varepsilon) = \diag \{ \frac{q^\varepsilon}{\rho^\varepsilon}, ~\bm{I}_3, ~ \alpha_1 \bm{\bar{A}}_0, ~\alpha_2, ~\frac{1}{\alpha_3}\bm{I}_3 \}, \quad \bm{\bar{A}}_0 = \diag\{1, ~ 2, ~ 2, ~1, ~2, ~ 1\}.
\end{equation}
Thus, the local existence theory for initial-value problems of symmetrizable hyperbolic systems (see Theorem 2.1 in \cite{majda1984compressible}) is applicable.

To proceed, we introduce the following notations.
\begin{notation}
	$\abs{\bm{U}}$ denotes some norm of a vector or matrix $\bm{U}$. For a nonnegative integer $k$, $H^k =H^k(\Omega)$ denotes the usual $L^2$-type Sobolev space of order $k$, with $\Omega=\mathbb{R}^3$ or the 3-dimensional torus. We write $\|\cdot\|_k$ for the standard norm of $H^k$ and $\|\cdot\|$ for $\|\cdot\|_0$. When $\bm{U}$ is a function of another variable $t$ as well as $x$, we write $\norm{\bm{U}(\cdot, t)}_k$ to recall that the norm is taken with respect to $x$ while $t$ is viewed as a parameter. In addition, we denote by $C([0, T], \mathbf{X})$ (resp. $\left.C^1([0, T], \mathbf{X})\right)$ the space of continuous (resp. continuously differentiable) functions on $[0, T]$ with values in a Banach space $\mathbf{X}$.
\end{notation}

Next, we denote by $G\subset \mathbb{R}\times\mathbb{R}^3\times\mathbb{R}^6\times\mathbb{R}_{+}\times\mathbb{R}^3$ the physical state space for our turbulence model \eqref{formal-full-equ-1}.
Assume $G_0$ is a convex compact subset of $G$,
 \begin{equation}
 	\bm{U}_0(x, \varepsilon)=(\phi_0(x, \varepsilon), \bm{u}_0(x, \varepsilon), \bm{\sigma}_0(x, \varepsilon), k_0(x, \varepsilon), \bm{y}_0(x, \varepsilon))\in G_0, \quad (x,\varepsilon) \in \Omega \times (0,1],
 \end{equation}
and $\bm{U}_0(\cdot, \varepsilon) \in H^s$ with $s>3/2+1$ an integer. 
Fix $\varepsilon$. For each convex open subset $G_1$ satisfying $G_0 \subset \subset G_1 \subset \subset G$, according to the local existence theory (Theorem 2.1 in \cite{majda1984compressible}) there exists $T>0$ so that the hyperbolic system \eqref{formal-full-equ-1} with initial data $\bm{U}_0(\cdot, \varepsilon)$ has a unique classical solution
\begin{equation}\nonumber
	\bm{U}^\varepsilon \in C\left([0, T], H^s\right) \quad \text { and } \quad \bm{U}^\varepsilon(x, t) \in G_1 \quad \forall(x, t) \in \Omega \times[0, T] .
\end{equation}
Define
\begin{equation}\nonumber
	T_\varepsilon=\sup \left\{T>0: \bm{U}^\varepsilon \in C\left([0, T], H^s\right), \quad \bm{U}^\varepsilon(x, t) \in G_1, \quad \forall(x, t) \in \Omega \times[0, T]\right\}.
\end{equation}
Namely, $\left[0, T_\varepsilon\right)$ is the maximal time interval of $H^s$ existence. Note that $T_\varepsilon$ depends on $G_1$ and may tend to zero as $\varepsilon$ goes to $0$.

To show that $T_\varepsilon$ has a positive lower bound, we refer to \cite{yong2005note} and introduce $\bm{U}_\varepsilon = (\phi_\varepsilon, ~ \bm{u}_\varepsilon, ~\bm{\sigma}_\varepsilon, ~k_\varepsilon, ~\bm{y}_\varepsilon)^T$ with 
\begin{equation}\label{equ:defin-u_epsilon}
	\phi_\varepsilon = \varepsilon \pi,\quad \bm{u}_\varepsilon = \bm{u},
	\quad \bm{\sigma}_\varepsilon = -\frac{\sqrt{\varepsilon}}{2 \rho_0}( l\sqrt{k} +\nu)\big(\nabla\bm{u} + (\nabla\bm{u})^T\big), \quad k_\varepsilon = k, \quad  \bm{y}_\varepsilon = - \frac{\sqrt{\varepsilon}\xi \alpha_3}{\rho_0}(l\sqrt{k}+\nu) \nabla k,
\end{equation}
where $(\bm{u}, ~\pi, ~k)$ is a smooth solution to the incompressible turbulence model \eqref{equ:rans-prandtl-one}. It is easy to verify that $\bm{U}_\varepsilon = (\phi_\varepsilon, ~ \bm{u}_\varepsilon, ~\bm{\sigma}_\varepsilon, ~k_\varepsilon, ~\bm{y}_\varepsilon)^T$ satisfies \eqref{equ:rans-prandtl-one} with remainders:
\begin{equation}\label{equ:4-4-fullequ_var}
\begin{aligned}
& \partial_t \phi_\varepsilon + \bm{u}_\varepsilon \cdot\nabla \phi_\varepsilon +  \frac{1}{\varepsilon q_\varepsilon} \nabla \cdot \bm{u}_\varepsilon ={} \varepsilon (\partial_t \pi + \bm{u} \cdot\nabla \pi)=:f_1  , \\[4mm]
&\partial_t \bm{u}_\varepsilon + \bm{u}_\varepsilon \cdot \nabla\bm{u}_\varepsilon + \frac{1}{\sqrt{\varepsilon}\rho_\varepsilon}\nabla\cdot \bm{\sigma}_\varepsilon + \frac{1}{\varepsilon \rho_\varepsilon} \nabla \phi_\varepsilon  + \frac{2}{3\rho_\varepsilon} \nabla k_\varepsilon ={}\brac{1-\frac{\rho_0}{\rho_\varepsilon}}(\partial_t \bm{u} + \bm{u} \cdot\nabla \bm{u}) =:f_2, \\[4mm]
&\partial_t \bm{\sigma}_\varepsilon +\bm{u}_\varepsilon \cdot \nabla \bm{\sigma}_\varepsilon + \frac{1}{\sqrt{\varepsilon} \alpha_1\rho_\varepsilon } \frac{1}{2}\bigg(\nabla\bm{u}_\varepsilon + (\nabla\bm{u}_\varepsilon)^T\bigg) +\frac{1}{\varepsilon \alpha_1}\frac{\bm{\sigma}_\varepsilon}{\nu_T +\nu} \\[2mm]
& \quad ={}-(\partial_t  +\bm{u}\cdot \nabla) \frac{\sqrt{\varepsilon}}{2\rho_0}(l\sqrt{k}+\nu)\big(\nabla\bm{u} + (\nabla\bm{u})^T\big) - \frac{1}{\sqrt{\varepsilon} \alpha_1}\brac{1-\frac{\rho_0}{\rho_\varepsilon}}\frac{1}{2 \rho_0}\big(\nabla\bm{u} + (\nabla\bm{u})^T\big)  =: f_3, \\[4mm]
&\partial_t k_\varepsilon + \bm{u}_\varepsilon \cdot \nabla k_\varepsilon +  \frac{2}{3\alpha_2\rho_\varepsilon}\nabla\cdot \bm{u}_\varepsilon  + \frac{\xi}{\sqrt{\varepsilon}\alpha_2\alpha_3\rho_\varepsilon} \nabla \cdot \bm{y}_\varepsilon 
- \frac{2\beta}{\varepsilon\alpha_2} \frac{\nu_T}{(\nu_T+\nu)^2} \bm{\sigma}_\varepsilon : \bm{\sigma}_\varepsilon 
+ \frac{\beta C_D}{\alpha_2} \frac{(k_\varepsilon)^{3/2}}{l} \\[2mm]
& \quad ={}\frac{1}{\rho_0} \brac{1-\frac{\rho_0}{\rho_\varepsilon}}\brac{\frac{\xi^2}{\alpha_2\rho_0} \nabla \cdot \brac{(\nu_T+\nu) \nabla k }} =:f_4, \\[4mm]
&\partial_t \bm{y}_\varepsilon + \bm{u}_\varepsilon \cdot \nabla \bm{y}_\varepsilon + \frac{\xi}{\sqrt{\varepsilon}\rho_\varepsilon}\nabla k_\varepsilon + \frac{1}{\varepsilon} \frac{\bm{y}_\varepsilon}{\alpha_3(\nu_T+\nu)} \\[2mm]
&\quad ={}-\sqrt{\varepsilon}(\partial_t + \bm{u} \cdot \nabla ) (\frac{\xi \alpha_3}{\rho_0} (\nu_T+\nu) \nabla k ) - \frac{\xi}{\sqrt{\varepsilon}}\brac{1-\frac{\rho_0}{\rho_\varepsilon}}\frac{\nabla k}{\rho_0}=: f_5.
\end{aligned}
\end{equation}
Here $q_\varepsilon = q(p_0 + \varepsilon^2 \pi)$ and $\rho_\varepsilon = \rho(p_0 + \varepsilon^2 \pi)$.
Namely, $\bm{U}_\varepsilon = (\phi_\varepsilon, ~ \bm{u}_\varepsilon, ~\bm{\sigma}_\varepsilon, ~k_\varepsilon, ~\bm{y}_\varepsilon)^T$ satisfies
\begin{equation}\label{equ:formal-limit-system}
	\begin{aligned}
		\bm{U}_{\varepsilon t} + \sum_{j=1}^3 \bm{A}_j(\bm{U}_\varepsilon)\bm{U}_{\varepsilon x_j} \equiv \bm{Q}(\bm{U}_\varepsilon) +  \bm{R}
	\end{aligned}
\end{equation}
with $\bm{R} = (f_1, ~f_2, ~f_3, ~f_4, ~f_5)^T$.

We are now ready to state our compatibility result.
\begin{theorem}
	Let $s>3/2+1$ be an integer. Suppose initial data $\bm{U}_0(x,\varepsilon)=(\phi_0(x, \varepsilon), \bm{u}_0(x, \varepsilon), \bm{\sigma}_0(x, \varepsilon), k_0(x, \varepsilon), \bm{y}_0(x, \varepsilon))$ belong to $H^s$ and satisfy
	\begin{equation}\nonumber
		\begin{aligned}
			&\norm{\brac{\phi_0(x, \varepsilon), ~\bm{u}_0(x, \varepsilon) - \bm{u}_0(x), ~k_0(x, \varepsilon)-k_0(x) }}_s = O(\varepsilon),\\[2mm]
			&\norm{\brac{\rho_0 \bm{\sigma}_0(x, \varepsilon)+\frac{\sqrt{\varepsilon}}{2}(l\sqrt{k_0}+\nu)\brac{\nabla\bm{u}_0 + (\nabla\bm{u}_0)^T}, ~\rho_0 \bm{y}_0(x, \varepsilon) +\sqrt{\varepsilon}\xi \alpha_3(l\sqrt{k_0}+\nu) \nabla k_0 }}_s = O(\varepsilon).
		\end{aligned}
	\end{equation}
	Let $(\bm{u}, ~\pi, ~k)$ be the smooth solution to system \eqref{equ:rans-prandtl-one} with initial data $\bm{w}(x, 0)=\bm{u}_0(x)$, $k(x, 0)=k_0(x)$ on $[0, T_\star]\times G$ and we assume
	\begin{equation}\nonumber
		(\bm{u}, ~\pi, ~k) \in C\left(\left[0, T_*\right], H^{s+1}\right) \cap C^1\left(\left[0, T_*\right], H^s\right).
	\end{equation}
	Then there exists a constant $\varepsilon_0>0$, such that, for all $\varepsilon < \varepsilon_0$, the hyperbolic turbulence system \eqref{formal-full-equ-1} with initial data $\bm{U}_0 (x, \varepsilon)$ has a unique smooth solution $\bm{U}^\varepsilon= (\phi^\varepsilon, ~ \bm{u}^\varepsilon, ~\bm{\sigma}^\varepsilon, ~k^\varepsilon, ~\bm{y}^\varepsilon)$. Moreover, there exists a positive constant $K>0$, independent of $\varepsilon$, such that, for all $\varepsilon\leq \varepsilon_0$,
		\begin{equation}\label{err:th1-err}
		\begin{aligned}
			&\norm{\brac{\phi^\varepsilon -\varepsilon \pi , ~\bm{u}^\varepsilon - \bm{u}, ~k^\varepsilon-k }}_s \leq K \varepsilon, \\[2mm]
			&\norm{\brac{\rho_0 \bm{\sigma}^\varepsilon+\frac{\sqrt{\varepsilon}}{2}(l\sqrt{k}+\nu)\big(\nabla\bm{u} + (\nabla\bm{u})^T\big),~\rho_0 \bm{y}^\varepsilon+\sqrt{\varepsilon}\xi \alpha_3(l\sqrt{k}+\nu) \nabla k}}_s \leq K  \varepsilon
		\end{aligned}
	\end{equation}
for $t\in [0, T_\star)$.
\end{theorem}

\begin{proof}

Following \cite{yong2005note} (see Theorem 2.2), it suffices to show that the error estimate \eqref{err:th1-err} for $t\in [0, \min\{T_\star, T_\varepsilon\})$. 
In this time interval, both $\bm{U}^\varepsilon$ and $\bm{U}_\varepsilon$ are regular enough and take values in a convex compact subset of the state space.

For the estimate, we set $\bm{E}=\bm{U}_\varepsilon-\bm{U}^\varepsilon$ and derive from \eqref{formal-full-equ-1} and \eqref{equ:formal-limit-system} that
\begin{equation}\nonumber
	\bm{E}_t+\sum_{j=1}^3 \bm{A}_j(\bm{U}^\varepsilon) \bm{E}_{x_j}=\sum_{j=1}^3 \big(\bm{A}_j(\bm{U}^\varepsilon)-\bm{A}_j(\bm{U}_\varepsilon) \big) \bm{U}_{\varepsilon x_j} + \bm{Q}(\bm{U}_\varepsilon) - \bm{Q}(\bm{U}^\varepsilon) + \bm{R}. 
\end{equation}
Differentiating this equation with $\partial^\alpha$ for any multi-index $\alpha$ satisfying $|\alpha| \leq s$ and setting $\bm{E}_\alpha=\partial^\alpha \bm{E}$, we get
\begin{equation}\label{err-equ-1}
	\bm{E}_{\alpha t}+\sum_{j=1}^3 \bm{A}_j (\bm{U}^\varepsilon) \bm{E}_{\alpha x_j}= \bm{F}_1^\alpha + \bm{F}_2^\alpha + \bm{G}^\alpha + \bm{R}_\alpha 
\end{equation}
with 
\begin{equation}\nonumber
	\begin{aligned}
		\bm{F}_1^\alpha = &{} \left\{\sum_{j=1}^3 \left[\bm{A}_j(\bm{U}^\varepsilon)-\bm{A}_j(\bm{U}_\varepsilon)\right] \bm{U}_{\varepsilon x_j}\right\}_\alpha, \\
		\bm{F}_2^\alpha = &{} \sum_{j=1}^3 \left\{\bm{A}_j(\bm{U}^\varepsilon) \bm{E}_{\alpha x_j}-\left[\bm{A}_j(\bm{U}^\varepsilon) \bm{E}_{x_j}\right]_\alpha\right\}, \\[2mm]
		\bm{G}^\alpha = &{} \left\{ \bm{Q}(\bm{U}_\varepsilon) - \bm{Q}(\bm{U}^\varepsilon) \right\}_\alpha.
	\end{aligned}
\end{equation}
Recall that $\bm{A}_0(\bm{U}^\varepsilon)$ and $\bm{A}_0(\bm{U}^\varepsilon) \bm{A}_j(\bm{U}^\varepsilon)$ are all symmetric.
We multiply the last equation by $2\bm{E}_{\alpha}^T \bm{A}_0(\bm{U}^\varepsilon)$ from the left and integrate the resulting equation over setting $x\in \mathbb{R}^3$ to obtain
\begin{equation}\label{err-equ-2}
	\begin{aligned}
		\brac{ \int \bm{E}_{\alpha}^T\bm{A}_0(\bm{U}^\varepsilon)\bm{E}_{\alpha} \diff x }_t={} &2 \int \bm{E}_{\alpha}^T \bm{A}_0(\bm{U}^\varepsilon)\brac{ \bm{F}_1^\alpha + \bm{F}_2^\alpha + \bm{G}^\alpha + \bm{R}_\alpha} \diff x \\[2mm]
		&+ \int  \bm{E}_{\alpha}^T \left \{ \frac{\partial }{\partial t}\bm{A}_0(\bm{U}^\varepsilon) + \sum_{j=1}^3 \frac{\partial }{\partial x_j} \brac{\bm{A}_0(\bm{U}^\varepsilon) \bm{A}_j(\bm{U}^\varepsilon)} \right \}\bm{E}_{\alpha} \diff x.
	\end{aligned}
\end{equation}

Now we estimate the terms in \eqref{err-equ-2}.  
From \eqref{defi:ao-symmetrizer} we see that $\bm{A}_0(\bm{U}^\varepsilon)=\bm{A}_0(\varepsilon \phi^\varepsilon)$ depends only on $\phi^\varepsilon$. Moreover, from \eqref{equ:Aj-define} we know that $\bm{A}_j(\bm{U}^\varepsilon)=u_j^\varepsilon \bm{I}_{14} + \bm{A}^{(1)}_j(\varepsilon \phi^\varepsilon) + \frac{1}{\sqrt{\varepsilon}} \bm{A}^{(2)}_j(\varepsilon \phi^\varepsilon) + \frac{1}{\varepsilon} \bm{A}^{(3)}_j(\varepsilon \phi^\varepsilon)$. 
Since $\bm{U}^\varepsilon$ is an exact solution to \eqref{formal-full-equ-1}, it is not difficult to get 
\begin{equation}\nonumber
\begin{aligned}
	& \frac{\partial }{\partial t}\bm{A}_0(\bm{U}^\varepsilon) + \sum_{j=1}^3 \frac{\partial }{\partial x_j} \brac{ \bm{A}_0(\bm{U}^\varepsilon) \bm{A}_j(\bm{U}^\varepsilon)} \\[2mm]
	={}&\frac{\partial }{\partial t}\bm{A}_0(\varepsilon \phi^\varepsilon) + \sum_{j=1}^3 \frac{\partial }{\partial x_j} \brac{u_j^\varepsilon\bm{A}_0(\varepsilon \phi^\varepsilon)  } + \sum_{j=1}^3 \frac{\partial }{\partial x_j} \brac{\bm{A}_0(\varepsilon \phi^\varepsilon) \brac{\bm{A}^{(1)}_j(\varepsilon \phi^\varepsilon) + \frac{1}{\sqrt{\varepsilon}} \bm{A}^{(2)}_j(\varepsilon \phi^\varepsilon) + \frac{1}{\varepsilon} \bm{A}^{(3)}_j(\varepsilon \phi^\varepsilon)} } \\[2mm]
	={}& \varepsilon \bm{A}_0'\brac{\frac{\partial }{\partial t} \phi^\varepsilon + \sum_{j=1}^3 u_j \frac{\partial }{\partial x_j} \phi^\varepsilon } + \bm{A}_0\nabla \cdot \bm{u}^\varepsilon \\[2mm]
	&\quad + \varepsilon \sum_{j=1}^3 \bigg[\bm{A}_0'\brac{\bm{A}^{(1)}_j + \frac{1}{\sqrt{\varepsilon}} \bm{A}^{(2)}_j+ \frac{1}{\varepsilon} \bm{A}^{(3)}_j} +  \bm{A}_0\brac{(\bm{A}^{(1)}_j)' + \frac{1}{\sqrt{\varepsilon}} (\bm{A}^{(2)}_j)' + \frac{1}{\varepsilon} (\bm{A}^{(3)}_j)'} \bigg] \frac{\partial \phi^\varepsilon}{\partial x_j} \\[2mm]
	={}& -\bm{A}_0' \frac{1}{q^\varepsilon}\nabla \cdot \bm{u}^\varepsilon + \bm{A}_0\nabla \cdot \bm{u}^\varepsilon 
	+ \sum_{j=1}^3 \bigg[ \bm{A}_0'\brac{\varepsilon \bm{A}^{(1)}_j + \sqrt{\varepsilon}\bm{A}^{(2)}_j+ \bm{A}^{(3)}_j} +  \bm{A}_0\brac{\varepsilon(\bm{A}^{(1)}_j)' + \sqrt{\varepsilon}(\bm{A}^{(2)}_j)' + (\bm{A}^{(3)}_j)'} \bigg] \frac{\partial \phi^\varepsilon}{\partial x_j}.
\end{aligned}
\end{equation}
Here the symbol $'$ denotes the differentiation with respect to $p^\varepsilon$. 
Thus we have
\begin{equation} \label{err-equ-4}
\begin{aligned}
	&\int  \bm{E}_{\alpha}^T \left \{ \frac{\partial }{\partial t}\bm{A}_0(\bm{U}^\varepsilon) + \sum_{j=1}^3 \frac{\partial }{\partial x_j} \brac{\bm{A}_0(\bm{U}^\varepsilon) \bm{A}_j(\bm{U}^\varepsilon)} \right \}\bm{E}_{\alpha} \diff x\\[2mm]
	\leq {}& C\brac{\abs{\nabla\cdot\bm{u}^\varepsilon} +\abs{\nabla \phi^\varepsilon} }\norm{\bm{E}}_{\abs{\alpha}}^2\\[2mm]
	\leq {}& C\brac{ \abs{\nabla\cdot (\bm{u}_\varepsilon-\bm{u}^\varepsilon)} +\abs{\nabla(\phi_\varepsilon-\phi^\varepsilon)}+ \abs{\nabla\cdot \bm{u}_\varepsilon} +\abs{\nabla \phi_\varepsilon} } \norm{\bm{E}}_{\abs{\alpha}}^2 \\[2mm]
	\leq {}& C\brac{1 + \norm{\bm{E}}_{s}}\norm{\bm{E}}_{\abs{\alpha}}^2.
\end{aligned}
\end{equation}
Here and below, C denotes a generic constant that can change from line to line.

For the other terms, we notice that 
\begin{equation}\label{err-equ-3}
\begin{aligned}
	2\int \bm{E}_{\alpha}^T \bm{A}_0(\bm{U}^\varepsilon)(\bm{F}_1^\alpha + \bm{F}_2^\alpha ) \diff x \leq {}& C\brac{ \norm{\bm{E}_{\alpha}}^2 + \norm{\bm{F}_1^\alpha}^2+ \norm{\bm{F}_2^\alpha}^2 }
\end{aligned}
\end{equation}
and estimate $\norm{\bm{F}_1^\alpha}$ and $\norm{\bm{F}_2^\alpha}$ with the help of the Moser-type inequalities in Sobolev spaces \cite{majda1984compressible}. For $\bm{F}_1^\alpha$, since
\begin{equation}\nonumber
\begin{aligned}
	\bm{A}_j(\bm{U}^\varepsilon)-\bm{A}_j(\bm{U}_\varepsilon) = (u_j^\varepsilon - u_{j\varepsilon})\bm{I}_{14} + \varepsilon (\phi^\varepsilon- \phi_\varepsilon)\brac{ (\bm{A}_j^{(1)})'(\ldots)+ \frac{1}{\sqrt{\varepsilon}} (\bm{A}_j^{(2)})'(\ldots )+\frac{1}{\varepsilon}(\bm{A}_j^{(3)})'(\ldots) }	
\end{aligned}
\end{equation}
with $(\ldots)$ being a convex combination of $\varepsilon \phi_\varepsilon$ and $\varepsilon(\phi^\varepsilon- \phi_\varepsilon)$, we use the boundedness of $\norm{\bm{U}_\varepsilon}_{s+1}$ to conclude
\begin{equation}\label{err-equ-5-1}
\begin{aligned}
	\norm{\bm{F}_1^\alpha}\leq {}&\sum_{j=1}^3 \norm{ \{ \brac{\bm{A}_j(\bm{U}^\varepsilon)-\bm{A}_j(\bm{U}_\varepsilon)} \bm{U}_{\varepsilon x_j} \}_\alpha }\\
	\leq {}& C\sum_{j=1}^3 \norm{\bm{U}_{\varepsilon x_j}}_s \norm{\bm{A}_j(\bm{U}^\varepsilon)-\bm{A}_j(\bm{U}_\varepsilon)}_{\abs{\alpha}}\\
	\leq {}& C \sum_{j=1}^3\bigg[\norm{u_{j\varepsilon}-u_j^\varepsilon }_{\abs{\alpha}} + \norm{ \phi_\varepsilon -\phi^\varepsilon}_{\abs{\alpha}}  \norm{\varepsilon (\bm{A}^{(1)}_j)'(\ldots) +\sqrt{\varepsilon} (\bm{A}^{(2)}_j)'(\ldots) + (\bm{A}^{(3)}_j)'(\ldots)}_s  \bigg]\\[2mm]
	\leq {}& C\brac{1 + \norm{\varepsilon \phi_\varepsilon + \theta \varepsilon(\phi_\varepsilon -\phi^\varepsilon) }_s^s} \norm{\bm{E}}_{\abs{\alpha}}\\[2mm]
	\leq {}& C\brac{1 +  \norm{\bm{E}}_s^s}\norm{\bm{E}}_{\abs{\alpha}}.
\end{aligned}	
\end{equation}
For $\bm{F}_2^\alpha$, we use the Moser-type calculus inequalities \cite{majda1984compressible}
\begin{equation}\nonumber
\begin{aligned}
	\norm{\bm{A}_j(\bm{U}^\varepsilon) \bm{E}_{\alpha x_j}-\left[\bm{A}_j(\bm{U}^\varepsilon) \bm{E}_{x_j}\right]_\alpha } {}&\leq C\norm{\partial \bm{A}_j}_{s-1}\norm{\bm{E}_{x_j}}_{\abs{\alpha}-1}
\end{aligned}	
\end{equation}
and
\begin{equation}\nonumber
\begin{aligned}
	\partial \bm{A}_j(\bm{U}^\varepsilon) =\partial u_j^\varepsilon \bm{I}_{14} + \varepsilon \partial \phi^\varepsilon\brac{(\bm{A}^{(1)}_j)'(\varepsilon \phi^\varepsilon) + \frac{1}{\sqrt{\varepsilon}}(\bm{A}^{(2)}_j)'(\varepsilon \phi^\varepsilon) + \frac{1}{\varepsilon} (\bm{A}^{(3)}_j)'(\varepsilon \phi^\varepsilon) }
\end{aligned}	
\end{equation}
to obtain
\begin{equation}\nonumber
\begin{aligned}
&\norm{\partial \bm{A}_j}_{s-1}\leq \norm{\partial u_j^\varepsilon}_{s-1} + C\norm{\partial \phi^\varepsilon }_{s-1}\norm{\brac{\varepsilon (\bm{A}^{(1)}_j)'(\varepsilon \phi^\varepsilon) + \sqrt{\varepsilon} (\bm{A}^{(2)}_j)'(\varepsilon \phi^\varepsilon) +  (\bm{A}^{(3)}_j)'(\varepsilon \phi^\varepsilon) }}_{s-1}.
\end{aligned}	
\end{equation}
Thus, we have
\begin{equation} \label{err-equ-5}
\norm{\bm{F}_2^\alpha} \leq C\brac{1 +\norm{\bm{E}}_{s}^{s}}\norm{\bm{E}}_{\abs{\alpha}}.
\end{equation}

Next, we follow \cite{yong1999singular} to estimate $\int  \bm{E}_{\alpha}^T \bm{A}_0(\bm{U}^\varepsilon) \bm{G}^\alpha \diff x$. Since
\begin{equation}\nonumber
\begin{aligned}
\bm{G}^\alpha = &{} \left\{ \bm{Q}(\bm{U}_\varepsilon) - \bm{Q}(\bm{U}^\varepsilon) \right\}_\alpha \\
= &{}\begin{pmatrix}
		0 \\ \bm{0}_{3\times 1} \\[2mm] 
		-\frac{1}{\varepsilon\alpha_1} \partial^\alpha \brac{\frac{\bm{\sigma}_\varepsilon}{l\sqrt{k_\varepsilon}+\nu} - \frac{\bm{\sigma}^\varepsilon}{l\sqrt{k^\varepsilon}+\nu}} \\[4mm] 
		\frac{2\beta l}{\varepsilon \alpha_2}\partial^\alpha \brac{ \frac{\sqrt{k_\varepsilon}}{(l\sqrt{k_\varepsilon}+\nu)^2}\bm{\sigma}_\varepsilon:\bm{\sigma}_\varepsilon -\frac{\sqrt{k^\varepsilon}}{(l\sqrt{k^\varepsilon}+\nu)^2}\bm{\sigma}^\varepsilon:\bm{\sigma}^\varepsilon}
		- \frac{\beta C_D}{\alpha_2l} \partial^\alpha \brac{\brac{(k_\varepsilon)^{3/2} - (k^\varepsilon)^{3/2} } }\\[3mm]
		-  \frac{1}{\varepsilon\alpha_3}\partial^\alpha \brac{\frac{\bm{y}_\varepsilon}{l\sqrt{k_\varepsilon}+\nu}- \frac{\bm{y}^\varepsilon }{l\sqrt{k^\varepsilon}+\nu} }
	\end{pmatrix}	
\end{aligned}	
\end{equation}
has three nonzero parts
and 
\begin{equation}\nonumber
	\bm{A}_0(\bm{U}^\varepsilon) = \diag \{ \frac{q^\varepsilon}{\rho^\varepsilon}, ~\bm{I}_3, ~ \alpha_1 \bm{\bar{A}}_0, ~\alpha_2, ~\frac{1}{\alpha_3}\bm{I}_3 \}, \quad \bm{\bar{A}}_0 = \diag\{1, ~ 2, ~ 2, ~1, ~2, ~ 1\},
\end{equation}
we write
\begin{equation}\nonumber
\begin{aligned}
	\int \bm{E}_{\alpha}^T \bm{A}_0(\bm{U}^\varepsilon) \bm{G}^\alpha \diff x ~\dot{=} ~ N_1 + N_2 + N_3.
\end{aligned}
\end{equation}
For $N_1$, we use the Moser-type inequalities, the positiveness of diagonal matrix $\frac{\bm{\bar{A}}_0}{l\sqrt{k}+\nu}$, and the definition  $k_\varepsilon =k$ and $\bm{\sigma}_\varepsilon =\sqrt{\varepsilon}O(1)$ in \eqref{equ:defin-u_epsilon} to deduce that
\begin{equation}\nonumber
\begin{aligned}
	N_1 ={}& -\frac{1}{\varepsilon} \int \partial^\alpha \brac{\bm{\sigma}_\varepsilon- \bm{\sigma}^\varepsilon }^T \bm{\bar{A}}_0 \partial^\alpha \brac{\frac{\bm{\sigma}_\varepsilon}{l\sqrt{k_\varepsilon}+\nu} - \frac{\bm{\sigma}^\varepsilon}{l\sqrt{k^\varepsilon}+\nu}}\diff x\\[4mm]
	={}& -\frac{1}{\varepsilon}\int \partial^\alpha\brac{\bm{\sigma}_\varepsilon- \bm{\sigma}^\varepsilon}^T \bm{\bar{A}}_0 \partial^\alpha \brac{\frac{\bm{\sigma}_\varepsilon-\bm{\sigma}^\varepsilon}{l\sqrt{k_\varepsilon}+\nu} }\diff x
	+\frac{1}{\varepsilon} \int \partial^\alpha \brac{\bm{\sigma}_\varepsilon- \bm{\sigma}^\varepsilon}^T \bm{\bar{A}}_0 \partial^\alpha \brac{\frac{\bm{\sigma}^\varepsilon}{l\sqrt{k^\varepsilon}+\nu} - \frac{\bm{\sigma}^\varepsilon}{l\sqrt{k_\varepsilon}+\nu}}\diff x\\[4mm]
	={}& - \frac{1}{\varepsilon}\int \partial^\alpha\brac{\bm{\sigma}_\varepsilon- \bm{\sigma}^\varepsilon}^T\frac{\bm{\bar{A}}_0}{l\sqrt{k_\varepsilon}+\nu}\partial^\alpha\brac{\bm{\sigma}_\varepsilon- \bm{\sigma}^\varepsilon}\diff x - \frac{1}{\varepsilon}\int \partial^\alpha \brac{\bm{\sigma}_\varepsilon- \bm{\sigma}^\varepsilon}^T\bm{\bar{A}}_0\sum_{0<\beta\leq \alpha} \partial^\beta \frac{1}{l\sqrt{k_\varepsilon}+\nu}\partial^{\alpha-\beta} (\bm{\sigma}_\varepsilon- \bm{\sigma}^\varepsilon)\diff x\\ 
	{}& +\frac{1}{\varepsilon}\int \partial^\alpha\brac{\bm{\sigma}_\varepsilon- \bm{\sigma}^\varepsilon}^T\bm{\bar{A}}_0\partial^\alpha (\frac{\bm{\sigma}^\varepsilon }{l\sqrt{k^\varepsilon}+\nu} - \frac{\bm{\sigma}^\varepsilon }{l\sqrt{k_\varepsilon}+\nu}) \diff x\\[4mm]
	\leq {}&- \frac{5\delta_1}{\varepsilon} \norm{\partial^\alpha( \bm{\sigma}_\varepsilon - \bm{\sigma}^\varepsilon)}^2 + \frac{C}{\varepsilon}\norm{ \partial^\alpha (\bm{\sigma}_\varepsilon - \bm{\sigma}^\varepsilon )}\norm{\frac{1}{l\sqrt{k_\varepsilon}+\nu}}_{s}\norm{\bm{\sigma}_\varepsilon- \bm{\sigma}^\varepsilon}_{\abs{\alpha}-1}  \\
	&+\frac{C}{\varepsilon}\norm{ \partial^\alpha (\bm{\sigma}_\varepsilon - \bm{\sigma}^\varepsilon )}\norm{\bm{\sigma}^\varepsilon }_s(1+\norm{k_\varepsilon -k^\varepsilon}_s^s)\norm{ k_\varepsilon -k^\varepsilon}_{\abs{\alpha}} \\[4mm]
	\leq {}&- \frac{5\delta_1}{\varepsilon} \norm{\partial^\alpha( \bm{\sigma}_\varepsilon - \bm{\sigma}^\varepsilon)}^2 + \frac{C}{\varepsilon}\norm{ \partial^\alpha (\bm{\sigma}_\varepsilon - \bm{\sigma}^\varepsilon )}\norm{\bm{\sigma}_\varepsilon- \bm{\sigma}^\varepsilon}_{\abs{\alpha}-1}  \\
	&+\frac{C}{\varepsilon}\norm{ \partial^\alpha (\bm{\sigma}_\varepsilon - \bm{\sigma}^\varepsilon )}\brac{\norm{\bm{\sigma}_\varepsilon}_s + \norm{\bm{\sigma}_\varepsilon -\bm{\sigma}^\varepsilon }_s }  (1+\norm{k_\varepsilon -k^\varepsilon}_s^s)\norm{ k_\varepsilon -k^\varepsilon}_{\abs{\alpha}} \\[4mm]
	\leq {}&- \frac{5\delta_1}{\varepsilon} \norm{\partial^\alpha( \bm{\sigma}_\varepsilon - \bm{\sigma}^\varepsilon)}^2 + \frac{C}{\varepsilon}\norm{ \partial^\alpha (\bm{\sigma}_\varepsilon - \bm{\sigma}^\varepsilon )}\norm{\bm{\sigma}_\varepsilon- \bm{\sigma}^\varepsilon}_{\abs{\alpha}-1}  \\
	&+\frac{C}{\varepsilon}\norm{ \partial^\alpha (\bm{\sigma}_\varepsilon - \bm{\sigma}^\varepsilon )}\brac{\sqrt{\varepsilon} + \norm{\bm{\sigma}_\varepsilon -\bm{\sigma}^\varepsilon }_s }  (1+\norm{k_\varepsilon -k^\varepsilon}_s^s)\norm{ k_\varepsilon -k^\varepsilon}_{\abs{\alpha}} \\[4mm]
	\leq {}&- \frac{3\delta_1}{\varepsilon} \norm{\partial^\alpha( \bm{\sigma}_\varepsilon - \bm{\sigma}^\varepsilon)}^2 + \frac{C}{\varepsilon} \norm{ \bm{\sigma}_\varepsilon - \bm{\sigma}^\varepsilon}_{\abs{\alpha}-1}^2  + C(1+\Delta(t)^2) (1 + \norm{\bm{E}}_s^{2s}) \norm{\bm{E}}_{\abs{\alpha}}^2.
\end{aligned}
\end{equation}
Here $\delta_1>0$ is the smallest eigenvalue of the positive definite matrix $\frac{\bm{\bar{A}}_0}{l\sqrt{k}+\nu}$ and
\begin{equation}\nonumber
\begin{aligned}
	\Delta(t) =  \frac{\norm{\bm{E}(t)}_{s} }{\sqrt{\varepsilon}}.
\end{aligned}	
\end{equation}
Similarly, we have
\begin{equation}\nonumber
\begin{aligned}
	N_3  ={}& -\frac{1}{\varepsilon\alpha_3^2} \int \partial^\alpha\brac{ \bm{y}_\varepsilon- \bm{y}^\varepsilon }^T \partial^\alpha \brac{\frac{\bm{y}_\varepsilon}{l\sqrt{k_\varepsilon}+\nu}- \frac{\bm{y}^\varepsilon }{l\sqrt{k^\varepsilon}+\nu}} \diff x\\[4mm]
	\leq {}&- \frac{3\delta_2}{\varepsilon} \norm{\partial^\alpha(\bm{y}_\varepsilon - \bm{y}^\varepsilon)}^2 + \frac{C}{\varepsilon} \norm{\bm{y}_\varepsilon - \bm{y}^\varepsilon}_{\abs{\alpha}-1}^2  + C(1+\Delta(t)^2)(1 + \norm{\bm{E}}_s^{2s}) \norm{\bm{E}}_{\abs{\alpha}}^2. 
\end{aligned}
\end{equation}
For $N_2$, we set
\begin{equation}\nonumber
\begin{aligned}
	F ~\dot{=}~ F(k, \bm{\sigma})={} \frac{2 \beta l\sqrt{k}}{(l\sqrt{k}+\nu)^2} \bm{\sigma} : \bm{\sigma} - \frac{\beta C_D}{l} k^{3/2},\\[2mm]
	k(\theta)=k_\varepsilon + \theta(k_\varepsilon - k^\varepsilon), \quad 
	\bm{\sigma}(\theta) = \bm{\sigma}_\varepsilon + \theta\brac{\bm{\sigma}_\varepsilon -\bm{\sigma}^\varepsilon}.
\end{aligned}	
\end{equation}
Then we deduce from 
\begin{equation}\nonumber
\begin{aligned}
	F(k_\varepsilon, \frac{\bm{\sigma}_\varepsilon}{\sqrt{\varepsilon}} ) - F(k^\varepsilon, \frac{\bm{\sigma}^\varepsilon}{\sqrt{\varepsilon}})
	= (k_\varepsilon - k^\varepsilon) \int_0^1 \partial_k F(k(\theta), \frac{\bm{\sigma}(\theta)}{\sqrt{\varepsilon}}) \diff \theta + \brac{\frac{\bm{\sigma}_\varepsilon}{\sqrt{\varepsilon}} -\frac{\bm{\sigma}^\varepsilon}{\sqrt{\varepsilon}}}^T\int_0^1 \partial_{\bm{\sigma}} F(k(\theta), \frac{\bm{\sigma}^\varepsilon}{\sqrt{\varepsilon}}) \diff \theta
\end{aligned}	
\end{equation} 
that
\begin{equation}\nonumber
\begin{aligned}
	N_2 ={}&\int  \partial^\alpha(k_\varepsilon - k^\varepsilon)\partial^\alpha \brac{F(k_\varepsilon, \frac{\bm{\sigma}_\varepsilon}{\sqrt{\varepsilon}} ) - F(k^\varepsilon, \frac{\bm{\sigma}^\varepsilon}{\sqrt{\varepsilon}})} \diff x\\[4mm]
	\leq {}& \norm{\partial^\alpha (k_\varepsilon - k^\varepsilon)}\norm{F(k_\varepsilon, \frac{\bm{\sigma}_\varepsilon}{\sqrt{\varepsilon}} ) - F(k^\varepsilon, \frac{\bm{\sigma}^\varepsilon}{\sqrt{\varepsilon}})}_{\abs{\alpha}} \\[4mm]
	\leq {}& C \norm{\partial^\alpha (k_\varepsilon - k^\varepsilon)}\brac{\norm{k_\varepsilon -k^\varepsilon}_{\abs{\alpha}}\norm{\partial_k F(k(\theta), \frac{\bm{\sigma}^\varepsilon}{\sqrt{\varepsilon}}) }_{s} + \norm{\frac{\bm{\sigma}_\varepsilon -\bm{\sigma}^\varepsilon }{\sqrt{\varepsilon}} }_{\abs{\alpha}}\norm{\partial_{\bm{\sigma}}F(k(\theta), \frac{\bm{\sigma}^\varepsilon}{\sqrt{\varepsilon}}) }_{s} } \\[4mm]
	\leq {}& C\norm{\partial^\alpha (k_\varepsilon - k^\varepsilon)}\brac{1 + \norm{k_\varepsilon -k^\varepsilon}_s^s + \norm{\frac{\bm{\sigma}_\varepsilon - \bm{\sigma}^\varepsilon}{\sqrt{\varepsilon}}}_s^s} \brac{\norm{k^\varepsilon -k_\varepsilon}_{\abs{\alpha}}+\norm{\frac{\bm{\sigma}_\varepsilon -\bm{\sigma}^\varepsilon }{\sqrt{\varepsilon}} }_{\abs{\alpha}} }  \\[4mm]
	\leq {}&C\brac{1 + \Delta(t)^s} \brac{ \norm{k_\varepsilon -k^\varepsilon }_{\abs{\alpha}}^2 + \norm{k^\varepsilon -k_\varepsilon}_{\abs{\alpha}}\norm{\frac{\bm{\sigma}_\varepsilon -\bm{\sigma}^\varepsilon }{\sqrt{\varepsilon}} }_{\abs{\alpha}} }.
\end{aligned}
\end{equation}
Consequently, we have 
\begin{equation}\label{equ:err-N123}
\begin{aligned}
	\int \bm{E}_{\alpha}^T \bm{A}_0(\bm{U}^\varepsilon) \bm{G}^\alpha \diff x \leq{}& - \frac{3\delta_1}{\varepsilon} \norm{\partial^\alpha( \bm{\sigma}_\varepsilon - \bm{\sigma}^\varepsilon)}^2  - \frac{3\delta_2}{\varepsilon} \norm{\partial^\alpha(\bm{y}_\varepsilon - \bm{y}^\varepsilon)}^2 \\[3mm]
	&+ \frac{C}{\varepsilon} \norm{ \bm{\sigma}_\varepsilon - \bm{\sigma}^\varepsilon}_{\abs{\alpha}-1}^2 + \frac{C}{\varepsilon} \norm{\bm{y}_\varepsilon - \bm{y}^\varepsilon}_{\abs{\alpha}-1}^2 \\[3mm]
	&+C\big(1 + \Delta(t)^s\big) \norm{\bm{E}}_{\abs{\alpha}}\norm{\frac{\bm{\sigma}_\varepsilon -\bm{\sigma}^\varepsilon }{\sqrt{\varepsilon}} }_{\abs{\alpha}} \\[3mm]
	&  + C(1+\Delta(t)^{s}) (1 + \norm{\bm{E}}_s^{2s}) \norm{\bm{E}}_{\abs{\alpha}}^2. 
\end{aligned}
\end{equation}

For the remaining term  $\int \bm{E}_{\alpha}^T \bm{A}_0(\bm{U}^\varepsilon)\bm{R}_\alpha \diff x$, we firstly deduce from the regularity of $(\bm{u}, ~\pi, ~k)$ that
\begin{equation}\nonumber
	1-\frac{\rho_0}{\rho_\varepsilon}=1-\frac{\rho(p_0)}{\rho(p_0 + \varepsilon^2 \pi)}=O(\varepsilon^2).
\end{equation}
Thus, it follows immediately from the expressions of $\bm{R}$ in \eqref{equ:4-4-fullequ_var} that 
\begin{equation}\nonumber
\begin{aligned}
	&\underset{t\in[0, T_\star]}{\max} \norm{f_i(\cdot, t)}_s \leq C \varepsilon, \qquad i=1,2,4,\\
	&\underset{t\in[0, T_\star]}{\max} \norm{f_i(\cdot, t)}_s \leq C \sqrt{\varepsilon} , \qquad i=3, 5,
\end{aligned}
\end{equation}
and
\begin{equation}\label{err-equ-3-2}
\begin{aligned}
	&\int \bm{E}_{\alpha}^T \bm{A}_0(\bm{U}^\varepsilon)\bm{R}_\alpha \diff x\\
	={} &\int  \frac{q^\varepsilon}{\rho^\varepsilon}\brac{\partial^\alpha\phi_\varepsilon-\partial^\alpha\phi^\varepsilon}\partial^\alpha f_1 \diff x + \int \brac{\partial^\alpha\bm{u}_\varepsilon-\partial^\alpha\bm{u}^\varepsilon}^T\partial^\alpha f_2  \diff x + \int \alpha_1  \brac{\partial^\alpha\bm{\sigma}_\varepsilon-\partial^\alpha\bm{\sigma}^\varepsilon}^T \bm{\bar{A}}_0\partial^\alpha f_3  \diff x\\[2mm]
	{}&+\int  \alpha_2 \brac{\partial^\alpha k_\varepsilon-\partial^\alpha k^\varepsilon}\partial^\alpha f_4  \diff x + \int \frac{1}{\alpha_3}\brac{\partial^\alpha\bm{y}_\varepsilon- \partial^\alpha\bm{y}^\varepsilon}^T\partial^\alpha f_5 \diff x\\[2mm]
	\leq {}& C \brac{\norm{\bm{E}_\alpha}^2 + \abs{\partial^\alpha f_1}^2  + \abs{\partial^\alpha f_2}^2 + \abs{\partial^\alpha f_4}^2 }\\[2mm]
	{}&+C \varepsilon\brac{\abs{\partial^\alpha f_3}^2+ \abs{\partial^\alpha f_5}^2} + \frac{\delta_1}{\varepsilon} \norm{\partial^\alpha( \bm{\sigma}_\varepsilon - \bm{\sigma}^\varepsilon)}^2 + \frac{\delta_2}{\varepsilon} \norm{\partial^\alpha( \bm{y}_\varepsilon - \bm{y}^\varepsilon)}^2\\[2mm]
	\leq {}& C \norm{\bm{E}_\alpha}^2  + \frac{\delta_1}{\varepsilon} \norm{\partial^\alpha( \bm{\sigma}_\varepsilon - \bm{\sigma}^\varepsilon)}^2 + \frac{\delta_2}{\varepsilon} \norm{\partial^\alpha( \bm{y}_\varepsilon - \bm{y}^\varepsilon)}^2 + C \varepsilon^2.
\end{aligned}
\end{equation}

Substituting \eqref{err-equ-4}--\eqref{err-equ-3-2} into \eqref{err-equ-2} yields
\begin{equation}\label{equ:err-equ-7}
\begin{aligned}
	&\frac{\diff}{\diff t} \int \bm{E}_{\alpha}^T\bm{A}_0(\bm{U}^\varepsilon)\bm{E}_{\alpha} \diff x + \frac{2\delta_1}{\varepsilon} \norm{\partial^\alpha( \bm{\sigma}_\varepsilon - \bm{\sigma}^\varepsilon)}^2  + \frac{2\delta_2}{\varepsilon} \norm{\partial^\alpha(\bm{y}_\varepsilon - \bm{y}^\varepsilon)}^2\\[4mm]
	 \leq {}& \frac{C}{\varepsilon} \norm{ \bm{\sigma}_\varepsilon - \bm{\sigma}^\varepsilon}_{\abs{\alpha}-1}^2 + \frac{C}{\varepsilon} \norm{\bm{y}_\varepsilon - \bm{y}^\varepsilon}_{\abs{\alpha}-1}^2 +C\big(1 + \Delta(t)^s\big) \norm{\bm{E}}_{\abs{\alpha}}\norm{\frac{\bm{\sigma}_\varepsilon -\bm{\sigma}^\varepsilon }{\sqrt{\varepsilon}} }_{\abs{\alpha}} \\[4mm]
	&  + C(1+\Delta(t)^{s}) (1 + \norm{\bm{E}}_s^{2s}) \norm{\bm{E}}_{\abs{\alpha}}^2 + C\varepsilon^2.
\end{aligned}
\end{equation}
Let $k$ be an integer such that $0\leq k\leq s$. We sum up the last inequality for all $\alpha$ with  $\abs{\alpha} \leq k$ and integrate \eqref{equ:err-equ-7} from 0 to $T$  to obtain
\begin{equation}\label{err:last-err-1}
\begin{aligned}
	&\sum_{\abs{\alpha}\leq k}\int  \bm{E}_{\alpha}^T\bm{A}_0(\bm{U}^\varepsilon)\bm{E}_{\alpha} \big|_{t=T} \diff x + \frac{2\delta_1}{\varepsilon} \int_0^T  \norm{ \bm{\sigma}_\varepsilon - \bm{\sigma}^\varepsilon}_k^2 \diff t + \frac{2\delta_2}{\varepsilon}\int_0^T \norm{\bm{y}_\varepsilon - \bm{y}^\varepsilon}_k^2 \diff t\\[4mm]
	\leq{}& \sum_{\abs{\alpha}\leq k} \int \bm{E}_{\alpha}^T\bm{A}_0(\bm{U}^\varepsilon)\bm{E}_{\alpha} \big|_{t=0}  \diff x +\frac{C}{\varepsilon} \int_0^T \norm{ \bm{\sigma}_\varepsilon - \bm{\sigma}^\varepsilon}_{k-1}^2\diff t + \frac{C}{\varepsilon}\int_0^T \norm{\bm{y}_\varepsilon - \bm{y}^\varepsilon}_{k-1}^2\diff t\\[4mm]
	&+C\int_0^T \big(1 + \Delta(t)^s\big) \norm{\bm{E}}_{k}\norm{\frac{\bm{\sigma}_\varepsilon -\bm{\sigma}^\varepsilon }{\sqrt{\varepsilon}} }_{k} \diff t +C \int_0^T(1+\Delta(t)^{3s})\norm{\bm{E}(t)}_{k}^2 \diff t +CT  \varepsilon^2 \\[4mm]
	\leq{} & \sum_{\abs{\alpha}\leq k} \int \bm{E}_{\alpha}^T\bm{A}_0(\bm{U}^\varepsilon)\bm{E}_{\alpha} \big|_{t=0}  \diff x +\frac{C}{\varepsilon} \int_0^T \norm{ \bm{\sigma}_\varepsilon - \bm{\sigma}^\varepsilon}_{k-1}^2\diff t + \frac{C}{\varepsilon}\int_0^T \norm{\bm{y}_\varepsilon - \bm{y}^\varepsilon}_{k-1}^2\diff t\\[4mm]
	&+\frac{\delta_1}{\varepsilon} \int_0^T  \norm{ \bm{\sigma}_\varepsilon - \bm{\sigma}^\varepsilon}_k^2 \diff t +C \int_0^T(1+\Delta(t)^{3s})\norm{\bm{E}(t)}_{k}^2 \diff t +CT_\star\varepsilon^2 .
\end{aligned}
\end{equation}
Thus we have
\begin{equation}\nonumber
\begin{aligned}
	&\frac{\delta_1}{\varepsilon} \int_0^T  \norm{ \bm{\sigma}_\varepsilon - \bm{\sigma}^\varepsilon}_k^2 \diff t + \frac{2\delta_2}{\varepsilon}\int_0^T \norm{\bm{y}_\varepsilon - \bm{y}^\varepsilon}_k^2 \diff t\\[4mm]	
	\leq{} & \sum_{\abs{\alpha}\leq k} \int \bm{E}_{\alpha}^T\bm{A}_0(\bm{U}^\varepsilon)\bm{E}_{\alpha} \big|_{t=0}  \diff x +\frac{C}{\varepsilon} \int_0^T \norm{ \bm{\sigma}_\varepsilon - \bm{\sigma}^\varepsilon}_{k-1}^2\diff t + \frac{C}{\varepsilon}\int_0^T \norm{\bm{y}_\varepsilon - \bm{y}^\varepsilon}_{k-1}^2\diff t\\[4mm]
	& +C \int_0^T(1+\Delta(t)^{3s})\norm{\bm{E}(t)}_{k}^2 \diff t +CT_\star\varepsilon^2 .
\end{aligned}
\end{equation}
A simple iteration based on this inequality and $\norm{\cdot}_{-1}=0$ yields
\begin{equation}\nonumber
\begin{aligned}
	&\frac{1}{\varepsilon} \int_0^T  \norm{ \bm{\sigma}_\varepsilon - \bm{\sigma}^\varepsilon}_k^2 \diff t + \frac{1}{\varepsilon}\int_0^T \norm{\bm{y}_\varepsilon - \bm{y}^\varepsilon}_k^2 \diff t\\[4mm]	
	\leq{} & C\sum_{\abs{\alpha}\leq k} \int \bm{E}_{\alpha}^T\bm{A}_0(\bm{U}^\varepsilon)\bm{E}_{\alpha} |_{t=0}  \diff x   +C \int_0^T(1+\Delta(t)^{3s})\norm{\bm{E}(t)}_{k}^2 \diff t +CT_\star\varepsilon^2 .
\end{aligned}
\end{equation}

From the last inequality and \eqref{err:last-err-1} with $k=s$, it follows that
\begin{equation}\nonumber
\begin{aligned}
	&\sum_{\abs{\alpha}\leq s}\int  \bm{E}_{\alpha}^T\bm{A}_0(\bm{U}^\varepsilon)\bm{E}_{\alpha}\big|_{t=T} \diff x \\[4mm]	
	\leq{} &C \sum_{\abs{\alpha}\leq s} \int \bm{E}_{\alpha}^T\bm{A}_0(\bm{U}^\varepsilon)\bm{E}_{\alpha} \big|_{t=0}  \diff x  +C \int_0^T(1+\Delta(t)^{3s})\norm{\bm{E}(t)}_{s}^2 \diff t +CT_\star\varepsilon^2.
\end{aligned}
\end{equation}
Note that $C^{-1} \norm{\bm{E}_{\alpha}} \leq \int \bm{E}_{\alpha}^T\bm{A}_0(\bm{U}^\varepsilon)\bm{E}_{\alpha}\diff x \leq C \norm{\bm{E}_{\alpha}}$ and $\norm{\bm{E}(0)}_s^2 =O(\varepsilon^2)$. We get
\begin{equation}\label{equ:err-equ-8}
	\norm{\bm{E}(T)}_s^2 \leq  C \int_0^T (1+\Delta(t)^{3s}) \norm{\bm{E}(t)}_s^2 \diff t + CT_\star \varepsilon^2.
\end{equation}
Applying Gronwall's lemma to \eqref{equ:err-equ-8} gives
\begin{equation}\label{equ:err-equ-9}
	\norm{\bm{E}(T)}_s^2 \leq CT_\star \varepsilon^2 \exp \brac{C \int_0^T(1+\Delta(t)^{3s}) \diff t } \equiv \varepsilon
	\Phi(T).
\end{equation}
Thus, we have
\begin{equation}\nonumber
	\Phi^{\prime}(t)=C (1+\Delta(t)^{3s}) \Phi(t) \leq C \Phi(t)+C \Phi^{1+3s/2}(t).
\end{equation}
Applying the nonlinear Gronwall-type inequality in \cite{yong1999singular} to this inequality yields
\begin{equation}\nonumber
	\Delta(t)^2 \leq  \Phi(t) \leq \exp \left(C T_\star\right)
\end{equation}
if $\Phi(0)=CT_\star \varepsilon < \exp (-C T_*)$. Because of \eqref{equ:err-equ-9}, there exists a constant $K$, independent of $\varepsilon$, such that
\begin{equation}\nonumber
	\norm{\bm{E}(t)}_s \leq K \varepsilon.
\end{equation}
This completes the proof.
\end{proof}

\section*{Acknowledgements}
Research is supported by National Key Research and Development Program of China (Grant no. 2021YFA0719200).


\end{CJK}
\end{document}